\documentclass[11pt]{amsart}

\usepackage{hyperref}
\usepackage{amssymb}
\usepackage{verbatim}

\makeindex

\usepackage{enumerate,color}

\renewcommand{\subset}{\subseteq}

\newcommand{\vertiii}[1]{{\left\vert\kern-0.25ex\left\vert\kern-0.25ex\left\vert #1 
    \right\vert\kern-0.25ex\right\vert\kern-0.25ex\right\vert}}

\def\CC{\mathbb{C}}
\def\NN{\mathbb{N}}
\def\RR{\mathbb{R}}
\def\HH{\mathbb{H}}

\def\cB{ {\mathcal B} }
\def\cC{ {\mathcal C} }

\def\cE{ {\mathcal E} }

\def\cH{ {\mathcal H} }
\def\cI{ {\mathcal I} }

\def\cR{{ \mathcal R }}
\def\cS{{\mathcal S} }
\def\cT{{\mathcal T}}

\def\cW{{\mathcal W}}

\def\x{x}
\def\xs{\x^{\ast}}
\def\ax{\langle \x \rangle}
\def\axs{\langle \x, \xs \rangle}

\newcommand{\matof}[3]{M_{#1}(#2)^{#3}}

\newcommand{\lm}[1]{T(#1)}
\newcommand{\lc}[1]{\operatorname{lc}(#1)}

\newtheorem{theorem}            {Theorem}[section]
\newtheorem{thm}           [theorem]{Theorem}
\newtheorem{cor}          [theorem]{Corollary}
\newtheorem{prop}        [theorem]{Proposition}

\newtheorem{lem}              [theorem]{Lemma}

\theoremstyle{definition}
\newtheorem{example}            [theorem]{Example}
\newtheorem{rem}             [theorem]{Remark}
\def\beq{ \begin{equation}}
\def\eeq{  \end{equation} }
\def\bes{\begin{equation*}}
\def\ees{\end{equation*}}

\def\ben{\begin{enumerate} }
\def\een{\end{enumerate} }

\def\bmat{\begin{bmatrix}}
\def\emat{\end{bmatrix}}
\def\barr{\begin{array}}
\def\earr{\end{array}}

\def\bet{\begin{theorem}
}
\def\eet{\end{theorem}}

\newcommand{\df}[1]{{\bf{#1}}{\index{#1}}}

\def\i{\index}

\def\fin{\mbox{\tiny finite}}

\def\lead{\mathrm{lead}}
\def\hom{{\mathrm{hom}}}
\def\hard{\mathrm{hard}}
\def\soft{\mathrm{soft}}
\def\lft{\mathrm{left}}
\def\real{\mathrm{real}}

\newcommand{\rr}[1]{\sqrt[\real]{#1}}

\newcommand{\ihard}[1]{\cI_\hard(#1)}
\newcommand{\vhard}[1]{V_\hard(#1)}
\newcommand{\rhard}[1]{\sqrt[\hard]{#1}}
\newcommand{\ichard}[2]{\cI_\hard^#1(#2)}
\newcommand{\vchard}[2]{V_\hard^#1(#2)}
\newcommand{\rchard}[2]{\sqrt[#1-\hard]{#2}}

\newcommand{\ileft}[1]{\cI_\lft(#1)}
\newcommand{\vleft}[1]{V_\lft(#1)}
\newcommand{\rleft}[1]{\sqrt[\lft]{#1}}
\newcommand{\icleft}[2]{\cI_\lft^#1(#2)}
\newcommand{\vcleft}[2]{V_\lft^#1(#2)}
\newcommand{\rcleft}[2]{\sqrt[#1-\lft]{#2}}

\newcommand{\isoft}[1]{\cI_\soft(#1)}
\newcommand{\vsoft}[1]{V_\soft(#1)}
\newcommand{\rsoft}[1]{\sqrt[\soft]{#1}}
\newcommand{\icsoft}[2]{\cI_\soft^#1(#2)}
\newcommand{\vcsoft}[2]{V_\soft^#1(#2)}
\newcommand{\rcsoft}[2]{\sqrt[#1-\soft]{#2}}

\def\chrisA{\mathcal A}

\def\F{\mathbb F}

\def\nss{Nullstellensatz}
\def\nsse{Nullstellens\" atze}
\def\FA{\mathbb \F \langle x,x^*\rangle}

\def\irr{\Pi_\mathrm{irr}}

\makeindex

\title[Real ideals in $*$-algebra]{Real Nullstellensatz and $*$-ideals in $*$-Algebras}

\author[Cimpri\v c, Helton, McCullough and Nelson]{Jakob Cimpri\v c${}^1,$ J. William Helton${}^2,$  Scott McCullough${}^3$ and Christopher Nelson${}^{2\dag}$}


\thanks{${}^1$Research supported by the grant P1--0222 from the Slovenian Research Agency – ARRS, ${}^2$Research supported by NSF grants
DMS-0700758, DMS-0757212, and the Ford Motor Co. and 
 ${}^3$Research supported by the NSF grant DMS-0758306.}

\begin{document}

\maketitle

%%new abstract

\begin{abstract}
 Let $\F$ denote either $\RR$ or $\CC$. An ideal $I$ in the free $\ast$-algebra $\F\axs$ in $g$ freely
 noncommuting variables $\{x_1,\dots,x_g\}$ and their formal
 adjoints $\{x_1^*,\dots,x_g^*\}$ is a $\ast$-ideal if $I=I^*$.  When a 
 real $\ast$-ideal has finite codimension, it satisfies a strong
  Nullstellensatz. Without the finite codimension assumption, there are examples of such ideals
  which do not satisfy, very liberally interpreted, any Nullstellensatz.
  A polynomial $p\in\F\axs$ is analytic if it is a polynomial
  in the variables $x_j$ only; that is if $p\in\F\ax$. 
  As shown in this article,  $\ast$-ideals generated
  by analytic polynomials do satisfy a natural Nullstellensatz
  and those generated by homogeneous analytic polynomials 
  have a particularly simple description.   The article 
  also connects the results here for $\ast$-ideals to 
  the literature on Nullstellensatz for left ideals in $\ast$-algebras generally
  and in $\F\axs$ in particular. It also develops the
  concomitant general theory of $\ast$-ideals in general $\ast$-algebras. 
\end{abstract}

%% old abstract
\iffalse
\begin{abstract}
Earlier work gave a real  \nss\ for left ideals in the free 
$*$-algebra of  noncommutative
polynomials, as well as results for
 general (not necessarily commutative) 
$*$-algebras. Now we turn to $*$- ideals  and 
give natural real \nsse  for them.
Of course every two-sided ideal $I$ is also a one sided ideal and 
we see  the situation of two-sided ideals as splitting in two parts.
In the case of free $*$-algebras, there is a 
definitive \nss\ when
$I$ is a finitely generated left ideal.  
When the ideal $I$ is not finitely generated 
one requires a more general notion of zero of a polynomial,
but even then
 of we give an example 
which refutes there being a simple picture. 
 In connection with this example, 
 some general theory of $*$-ideals in a $*$-algebra is developed.
This paper gives several natural situations of two-sided ideals
possessing an elegant \nss. ## needs work. is the split finitely
 generated or finite codimension? ## 
\end{abstract}
\fi
%% end old abstract

\section{Introduction}
Let $\F$ be either $\RR$ or $\CC$ with complex conjugation as involution.
Let $\axs$ be the monoid freely generated by $\x=(x_1,\ldots, x_g)$ and 
$\xs=(x_1^*,\ldots,x_g^*)$, i.e., $\axs$ consists of  \df{words} in the $2g$
noncommuting letters $x_{1},\ldots,x_{g},x_1^*,\ldots,x_g^*$
(including the empty word $\emptyset$ which plays the role of the identity $1$).
Let $\FA$ denote the $\F$-algebra freely generated by $\x,\xs$, i.e., the elements of $\FA$
are \df{polynomials} in the noncommuting variables $\x,\xs$ with coefficients in $\F$.  
Equivalently, $\FA$ is the \df{free $\ast$-algebra} on $\x$.   Elements of the
 free algebra $\F\ax$ generated by $\x=(x_1,\dots,\x_g)$ are known
 as \df{analytic polynomials}.  A polynomial is \df{homogeneous}
  if it is an $\F$ linear combination of words of the same length. 

A left ideal $I\subseteq\F\axs$ is a $\ast$-\df{ideal} if $I^*=I$ and it is not hard to see that
such an ideal must also be a two-sided ideal. 
For $*$-ideals in $\FA$, there is a major distinction between those
that have finite codimension and those that have not.
In the first case there is a very strong real \nss\
whereas in the second case even a very weak version of real \nss\ fails in general.
In the positive direction, we show, independent of any finiteness hypotheses,
  that if the $\ast$-ideal 
 is generated by analytic polynomials, then it automatically satisfies
 a natural Nullstellensatz; and if it generated by analytic homogeneous
  polynomials, then it  has an especially simple representation.

 The body of the paper is organized as follows.  
 Section \ref{sec:realIdeals} presents basic properties of real ideals,
 including the Nullstellensatz in the case the real $\ast$-ideal $I$ has finite codimension. 
 Negative examples which illustrate that, absent additional hypotheses on $I$,
 a general Nullstellensatz is problematic are in Section \ref{sec:mainex}. There
  too is some needed additional theory applicable to general $\ast$-algebras. 
  The results for $\ast$-ideals generated by analytic polynomials
  are in Section \ref{sec:analytic}.   The articles \cite{chmn} and \cite{chkmn} gave real \nss\
for left ideals in $\FA$ and more general (not necessarily commutative, but associative) $*$-algebras.
 The relationship between the Nullstellensatz  in those papers for left ideals and
 those in this article for $\ast$-ideals are outlined in Section \ref{sec:leftZeroes}.
 Section \ref{sec:softZeroes} contains results for soft zero sets,
 namely, for $\{X \colon \det p(X) =0 \}$.

In the remainder of this introduction, we state our main results precisely, introducing notations 
and terminology as needed.  The focus initially is on the concrete example $\FA$.

\subsection{Ideals in $\FA$}
 \label{sec:idealsinFA}
Let $\chrisA$ be a unital associative $\F$-algebra with involution $\ast$, or \df{$\ast$-algebra} for short. 
A left ideal in the $*$-algebra $\chrisA$ is \df{real} if  $a_1,\dots,a_n \in \chrisA$ and 
\[
 \sum a_j^* a_j \in I+I^*
\]
implies $a_j\in I$ for each $j$. 
 A two-sided ideal is real if it is real as a left ideal.
 Moreover, as seen in Lemma \ref{lem:zi}\eqref{it:zi1}, a two-sided real ideal is in fact a $\ast$-ideal.

The \df{real radical},   denoted $\rr{I}$ of  a left ideal $I$
\i{$\rr{I}$ }
is the intersection of all real left ideals containing $I$ or equivalently, the smallest real left ideal containing $I$.  
The real radical of a $\ast$-ideal is also a $\ast$-ideal (see Lemma \ref{lem:zi}\eqref{it:zi2}).
%%%We will also give an iterative description of the real radical.
 
When $\chrisA =\FA$ there is a natural way to generate real $*$-ideals.  Given a positive integer $n$, let
$\matof{n}{\F}{g}$ denote the set of $g$-tuples $X=(X_1,\dots,X_g)$ of $n\times n$ matrices. 
Let $\matof{}{\F}{g}$ denote the graded set $(\matof{n}{\F}{g})_n.$
An element $p\in \FA$ is naturally evaluated at $X$ by substituting  $X_j$ for $x_j$ and the 
adjoint $X_j^*$ for $x_j^*$ with the result $p(X)$ being  an $n\times n$ matrix.
We say that $X$ is a \df{(hard) zero} of $p$ if $p(X)=0$.

Given a sequence $S=(S_n)_n$  of subsets $S_n$ of $\matof{n}{\F}{g}$, define its \df{hard vanishing set}
\[
 \ihard{S} = \{p\in\FA \colon p(X)=0 \mbox{ for every } n
  \mbox{ and every } X\in S_n \}.
\]
\index{$ \ihard{S} $}
 It is easily checked that $\ihard{S}$ is indeed  a real $*$-ideal.
 Moreover, if $S$ is a finite set, then the dimension of
  $\FA/\ihard{S}$ is finite. 
  
  The connection with \nsse\ is the following.  
  The \df{hard variety} \i{$\vhard{I} $}
  $\vhard{I}=(\vhard{I}_n)$ of an ideal $I$ in $\FA$
  is the sequence
\[
  \vhard{I}_n =\{X\in\matof{n}{\F}{g} \colon p(X)=0 \mbox{ for every } p\in I\}.
\]
  The \df{hard radical} of $I$ is
   \i{$\rhard{I}$}
\[
 \rhard{I} = \ihard{\vhard{I}},
\]
 which is necessarily a $*$-ideal. Finally, 
  the $*$-ideal has the \df{\nss \  property} if
\[
   \rhard{I}  = I
\]
  and $I$ satisfies the \df{real \nss \ } if 
\[
 \rhard{I}=\rr{I}.
\]
  
  %Basic properties are:
 When the codimension of $I$ in $\F\axs$ if finite, 
 the relation between real Nullstellensatz and real ideals
 is clean, readily described and essentially a consequence
  of the existing theory of formally real $\ast$-algebras.

\begin{prop} 
\label{thm:mainfindim}
 Let $I \subset \F\axs$ be a two-sided ideal.
\begin{enumerate}[(i)]
% \item 
%   $\rr{I}$ is the smallest two-sided real ideal containing $I$; 
% \item If $I$ is real, then $I$ is a $\ast$-ideal; 
\item\label{it:finGen}
$I$ is finitely generated as a left ideal if and only if either $I = \{0\}$
 or $I$ has finite codimension in $\F\axs$. %$\dim(\F\axs/I) < \infty;$ and
 \item 
  \label{it:big}
  If I is a  $*$-ideal, 
 then $I$ is real and  $0 < \dim(\F\axs/I) < \infty$
 if and only if there exists
 an $n \in \NN$ and a $X \in
\matof{n}{\F}{g}$ such that $I = \ihard{\{X\}}$.
  In particular, if $I$ is an ideal in $\FA$
  and $0<\dim(\F\axs/\rr{I})<\infty$, then
  $\rhard{I}=\rr{I}.$
\end{enumerate}
\end{prop}

\begin{rem}
 Thus if a $*$-ideal $I$ has finite codimension, then 
  $I$ has the Nullstellensatz property if and only if 
  it is real.

  The article contains two proofs of Proposition \ref{thm:mainfindim}. 
  The first is based on a standard applications of the well developed theory of formally real
  $\ast$-algebras. Another proof is an easy consequence of the theory developed here
  in Section \ref{sec:leftZeroes}
  to connect Nullstellensatz for left ideals for those for two-sided and $\ast$-ideals. 
 % to understand the situation for ideals $I$ of infinite codimension.   
\qed \end{rem}

\subsection{The Case of Infinite Codimension} 
 \label{sec:infinitecodim} 
%Now we move into the case where $I$ has infinite codimension.
% Let $\chrisA=\F\langle x,x^* \rangle$ denote the free $\ast$-algebra in one
%variable $x$ and
% let $I$ be the $\ast$-ideal of $\chrisA$ generated by $1-x^\ast x$.
   The \df{(free) Toeplitz algebra} $\cT$ is the quotient of $\F\langle x,x^*\rangle$ 
  by the $\ast$-ideal $I$ generated by $1-x^\ast x$. 
  It turns out  $I$ is real but $1-xx^\ast \in\rhard{I} \setminus I$
  and hence $I$ does not satisfy the  \nss\ property.
  Thus $I$ provides an example
  which shows that the finite codimension hypothesis is needed in
  Proposition \ref{thm:mainfindim} \eqref{it:big}.  
  The details can be found in Section \ref{sec:mainex}.
  A more elaborate example provided by
  Proposition \ref{prop:jaka} below shows that there are
  real ideals $I$ in $\FA$ which, even with a most liberal interpretation,
  do not satisfy a Nullstellensatz.

In spite of these negative examples, the main results of this article 
show that natural Nullstellensatze hold for ideals generated by
analytic polynomials, thus providing optimism that a satisfying
general theory may emerge.

\begin{cor}
\label{cor:genByAnalReal1}
If $I \subset \F\axs$ is a $\ast$-ideal generated by
analytic polynomials, then 
$$\rr{I} = I.$$
 In particular, $I$ is real. 
\end{cor}

\begin{proof}
This corollary  is an immediate consequence of Corollary \ref{cor:genByAnalReal2}.
\end{proof}

\begin{thm}
\label{thm:homAnalHilb}
 If  $I \subsetneq \F\axs$ is a homogeneous $\ast$-ideal generated by
 analytic polynomials, then there exists a Hilbert space $\cH$ and a tuple of bounded
 operators $X$ on $\cH$ such that $p(X) = 0$ if and only if $p \in I$. 
 Thus in an operator theoretic sense $I=\ihard{X}$. 
 In particular, $I$ has the Nullstellensatz property.
\end{thm}

\begin{proof}
See the end of \S \ref{sec:analyticMain}.
\end{proof}

\subsection{General $\ast$-Algebras}
 % Also we shall investigate    
 % a very liberal notion of the variety of an ideal
 % in a general setting.
 % While some properties are canonical and appealing,
 %  we shall show 
 % there are real $*$-ideals without
 % the Nullstellensatz property. This will now be explained.  
 Even allowing for a liberal notion of zero, and hence of 
  variety of an ideal, there are real $\ast$-ideals without
  the Nullstellensatz property as we soon explain.
  While the example given here is an ideal
  in a free $\ast$-algebra, the natural context
  for much of the discussion is that 
  of a general (associative)  $\ast$-algebra $\chrisA$ over $\F \in \{\RR,\CC\}$. 
  Its elements will be considered as noncommutative polynomials.
  
  A \df{$\ast$-representation} $\pi$ of $\chrisA$ is a unital $\ast$-homomorphism from $\chrisA$ to the $\ast$-algebra
  of all adjointable operators on some pre-Hilbert space $V_\pi$ over $\F$.
  %For such a $\ast$-representation $\pi$ we write $V_\pi$ for the corresponding pre-Hilbert space.
  Let $\cR$ be the class of all $\ast$-representations of the $\ast$-algebra $\chrisA$ and let 
  $\cC$ be a fixed subclass of $\cR$ whose elements will be considered as (evaluations at) real points.
  We say that a \df{real point} $\pi \in \cC$ is a \df{hard zero}
   of a {\it polynomial} $a \in \chrisA$ if $\pi(a)=0$.
  For a subset $T$ of $\cC$, let
  \[
  \ichard{\cC}{T}:= \{a \in \chrisA \colon \pi(a)=0 \mbox{ for every } \pi \in T\} 
  \]
  \i{$\ichard{\cC}{T}$}
  be its \df{hard vanishing set}.  
  For a subset $\cS$ of $\chrisA,$ let
  \[
  \vchard{\cC}{\cS} := \{\pi \in \cC \colon \pi(s)=0 \mbox{ for every } s \in \cS\} 
  \]
  \i{$ \vchard{\cC}{\cS}$}
  be its \df{hard variety} and  
  \[
  \rchard{\cC}{\cS} := \ichard{\cC}{\vchard{\cC}{\cS}}
  \]
  \i{ $\rchard{\cC}{\cS} $}
   its \df{hard radical}. 
      If $I(\cS)$ is the $\ast$-ideal of $\chrisA$ 
      generated by $\cS$, then clearly
$$\vchard{\cC}{\cS}=\vchard{\cC}{I(\cS)} \quad \text{ and } \quad \rchard{\cC}{\cS}=\rchard{\cC}{I(\cS)}.$$
\i{$\vchard{\cC}{\cS}$}
The relation between a $\ast$-ideal  $I$ 
and its various radicals is summarized
 by (see Proposition \ref{prop:basicin})
\[
 I\subset \rr{I} \subseteq \rchard{\cR}{I} \subseteq \rchard{\cC}{I}.
\]
   We say that a $\ast$-ideal $I$ \df{satisfies the real Nullstellensatz 
   over $\cC$} if $\rchard{\cC}{I}=\rr{I}$. We say that $I$ has the 
\df{Nullstellensatz property over $\cC$} if $\rchard{\cC}{I}=I$
(which implies that $I$ is real.)

\begin{example}
\label{ex:fini}
If $\chrisA=\F \axs$ and $\Pi$ is the class of all finite-dimensional $\ast$-representations,
then $\rchard{\Pi}{I}=\rhard{I}$ for every $\ast$-ideal $I$ of $\chrisA$. Here we identify 
every $\pi \in \Pi$ with the $g$-tuple $(\pi(x_1),\ldots,\pi(x_g)) \in \matof{}{\F}{g}$.
Therefore, $I$ satisfies the real Nullstellensatz over $\Pi$ if and only if it satisfies the (hard) real Nullstellensatz.
\qed \end{example}

Motivated by Example \ref{ex:fini} we introduce the following abbreviations when $T \subseteq \Pi$ and $\cS \subseteq \chrisA$:
$$\ihard{T} = \ichard{\Pi}{T}, \quad \vhard{\cS}=\vchard{\Pi}{\cS} \quad \mbox{and} \quad \rhard{\cS}=\rchard{\Pi}{\cS}.$$

\begin{example}
\label{ex:matpoly}
Let $\chrisA=\matof{n}{\F[x]}{}$ be the algebra of all polynomials in
commuting variables $x=(x_1, \cdots, x_g)$
with coefficients in $n \times n$ matrices over $\F$. The involution is trivial
on variables and it is hermitian transpose on coefficients. 
Let $\cE$ \i{$\cE$ }
be the class of all $n$-dimensional $\ast$-representations
(i.e. all evaluations at real points from $\F^g$).
By \cite[Corollary 18]{c1}, every $\ast$-ideal of $\chrisA$ satisfies the real
Nullstellensatz over $\cE$.
The case $n=1$ corresponds to the classical Real Nullstellensatz
\cite{dub,ris,efr}.
\qed \end{example}

The following proposition, based on an example introduced in \cite{pop1},  shows that
% absent the finite generation hypothesis,
a general Nullstellensatz is highly problematic.

\begin{prop}
 \label{prop:jaka}
 Fix  $0<q<1$. % and let $$\chrisA=\F\langle a,x,a^\ast,x^\ast \rangle$$
  The $\ast$-ideal in the  free $\ast$-algebra $\F\langle a,x,a^*,x^*\rangle$ in the two variables $a$ and $x$
 generated by
\[
  a^\ast a-q a a^\ast \quad \text{and} \quad xx^\ast+aa^\ast-1
\]
 is real but it does not satisfy the Real Nullstellensatz 
 over any representation class.
\end{prop}

 The proof is based upon results of \cite{pop1} and some general
 theory of $*$-algebras developed here. The details are in Section \ref{sec:mainex}

\section{Properties of Real Ideals}
\label{sec:realIdeals}
In this section the basic properties of real ideals, including the proof 
 of Proposition \ref{thm:mainfindim},  are collected. % for later use,  in this section.
% This section 
%we prove a few basic properties of real ideals.
% In the course of doing this we prove Theorem \ref{thm:mainfindim}.

%% comment ends about 10 lines below
\begin{comment}
A left ideal $I$ of a $\ast$-algebra $\chrisA$ is \df{real} if and only if
whenever
\[ \sum_i^{\fin} p_i^*p_i \in I + I^*\]
then each $p_i \in I$.
A two-sided ideal $J \subset \chrisA$ is real if it is a real left ideal.

Define the \df{real radical} of a left ideal $I$ to be
\[\rr{I} = \bigcap_{\substack{J \mbox{\tiny\  real left ideal}\\I
\subset J}} J = 
\mbox{the smallest real left ideal containing } I.
\] \i{$\rr{I}$}
\end{comment}
%%% comment begins about 10 lines above.

\begin{lem}
 \label{lem:zi0}
   If $I$ is a left ideal in the $\ast$-algebra $\chrisA,$ then
   there is a largest two-sided ideal $Z(I)$ contained in $I$.
   Indeed, $Z(I)$ is the kernel of the left regular representation of
    $\chrisA$ on $\chrisA/I$. 
   Moreover, if $I$ is real, then  $Z(I)$ is  real.
\end{lem}

\begin{proof}
  The kernel $Z$ of the left regular representation is a two-sided ideal contained in $I$.
  On the other hand, if $J\subset I$ is a two-sided ideal,
  $\theta\in J$ and $v \in \chrisA$, then $\theta v \in J\subset I$
  and hence $\pi(\theta)=0$ and 
   $\theta\in Z$.   As an aside, note that 
\[
  Z = \{ \vartheta \in I \colon \vartheta p \in
   I \mbox{ for every } p \in \chrisA\}.
\]
% By construction $Z$ is a two-sided ideal and $Z\subset I$.  On the other hand,
% if $J\subset I$ is a two-sided ideal and $\vartheta \in J$, then $\vartheta p \in J\subset I$
%  for each $p\in\chrisA$. Thus, $J\subset Z$.  Hence there is a largest two-sided
%  ideal contained in $I$. 

 Now suppose $I$ is a  real ideal. To see that $Z$ is a real ideal, suppose
\[
  \sum_i^{\fin} p_i^*p_i \in Z  + Z^*.
\]
 Since $I$ is real, each $p_i \in I$.   Further, 
 if $a,b \in Z$, then
 $q(a + b^*)r = (qar) + (r^*bq^*)^* \in Z +Z^*,$
   for each $q,r \in \chrisA.$
 Therefore for each $q \in \chrisA$,
\[
  \sum_i^{\fin} q^*p_i^*p_iq \in Z +Z^*,
\]
 which implies that each $p_iq \in I$.  Therefore each $p_i \in Z.$
 Hence $Z$ is a real ideal. 
\end{proof}

\begin{lem}
\label{lem:zi}
 Let $I \subset \chrisA$ be a two-sided ideal in the 
 $*-$algebra $\chrisA$.
\begin{enumerate}[(i)]
 \item \label{it:zi1} If $I$ is real, then $I = I^*$.
 \item \label{it:zi2} The radical $\rr{I}$ is the smallest two-sided real ideal containing $I$.
\end{enumerate}
\end{lem}

\begin{proof}
First, if $I$ is real, then for each $\iota \in I$, we have $\iota\iota^* \in
I$ since $I$ is two-sided. Thus $\iota^* \in I$ since $I$ is real.
Therefore $I = I^*$.

As in Lemma \ref{lem:zi0}, let $Z(\rr{I})$ be the largest two-sided ideal of $\chrisA$ contained in
$\rr{I}.$, % i.e.,
%\[Z(\rr{I}) := \{ \vartheta \in \rr{I} \colon \vartheta p \in
%\rr{I} \mbox{ for every } p \in \chrisA\}.\]
Since $I \subset \rr{I}$, and $I$ is a two-sided ideal, $I \subset
Z(\rr{I})$.
%We now show that $Z(\rr{I})$ is a real ideal.
%Suppose
%\[\sum_i^{\fin} p_i^*p_i \in Z(\rr{I}) + Z(\rr{I})^*.\]
%Then since $\rr{I}$ is real, each $p_i \in \rr{I}$.   We see
%also that if $a,b \in Z(\rr{I})$, then for each $q,r \in \chrisA$,
%$q(a + b^*)r = (qar) + (r^*bq^*)^* \in Z(\rr{I}) +
%Z(\rr{I})^*$. Therefore for each $q \in \chrisA$,
%\[\sum_i^{\fin} q^*p_i^*p_iq \in Z(\rr{I}) + Z(\rr{I})^*,\]
%which implies that each $p_iq \in \rr{I}$.  Therefore each $p_i \in
%Z(\rr{I})$.
Thus $Z(\rr{I}) \subset \rr{I}$ is a real left ideal
containing $I$.  Hence $Z(\rr{I}) = \rr{I}$.
Since all two-sided ideals are also left ideals, and $\rr{I}$ is the
smallest real left ideal containing $I$, it must also be the smallest real
two-sided ideal containing $I$.
\end{proof}

\begin{prop} 
 \label{prop:basicin}
If $I$ is a $\ast$-ideal of $\chrisA,$ then 
$$I\subset \rr{I} \subseteq \rchard{\cR}{I} \subseteq \rchard{\cC}{I}.$$
\end{prop}

\begin{proof}
It is clear that the kernel of a $\ast$-representation is always a real
$\ast$-ideal. In particular,
$\ichard{\cC}{T}$ is a real $\ast$-ideal for every subset $T$ of $\cC$. For $T=\vchard{\cC}{I}$
we get that
$\rchard{\cC}{I}$ is a real $\ast$-ideal. Since $\rchard{\cC}{I}$ contains $I$, it
follows that
$\rr{I} \subseteq \rchard{\cR}{I}$. The third inclusion is clear from $\cC
\subseteq \cR$.
\end{proof}

\subsection{Formally Real $\ast$-Algebras}
 A $\ast$-algebra $\chrisA$ is \df{formally real} (or \df{very
proper})
if $a_1,\ldots,a_k \in \chrisA$ and  $\sum_{i=1}^k a_i^\ast a_i
=0$ implies $a_1=\ldots=a_k=0$.
 Let  $\Sigma_\chrisA$ denote the set of all finite sums of elements $a^\ast a$, $a
\in \chrisA$.
 Alternately,  $\chrisA$ is formally real if and only if it is proper (i.e. $a^\ast
a=0$ implies $a=0$ for every
$a \in \chrisA$) and $-\Sigma_\chrisA \cap \Sigma_\chrisA={0}$ (i.e. 
for every $a_1,\ldots,a_k \in \chrisA$ such that $\sum_{i=1}^k a_i^\ast a_i =0$
we have that 
$a_i^\ast a_i =0$ for all $i$.)

Let $I_h=\{a \in I \colon a^*=a\}=I \cap \chrisA_h$ denote
 the \df{hermitian} elements of $I$. 

\begin{lem}
\label{fr}
For a $\ast$-ideal $I$ of a $\ast$-algebra $\chrisA$, the following are
equivalent.
\begin{enumerate}
\item $I$ is real; i.e., $I=\rr{I};$ 
\item Both $(\Sigma_A+I_h) \cap -(\Sigma_A+I_h)=I_h$ and $a^\ast a \in I$ implies $a
\in I$ for every $a \in \chrisA$,
\item The quotient $\chrisA/I$ is formally real.
\end{enumerate}
%$I_h=\{a \in I \colon a^*=a\}=I \cap \chrisA_h$.
\end{lem}

\begin{rem}
 \label{rem:Jaka-procedure}
   There is an iterative description of 
  $\rr{I}$ along the lines of 
  that for the real radical of a left ideal
  as described in \cite[Section 5]{chmn}.
\qed \end{rem}

\iffalse
%%-- iffalse ends 30 lines below
\subsection{An Iterative Description of $\rr{I}$}
In \cite[Section 5]{chmn} we gave an iterative description of $\rr{I}$
for a left ideal $I$. 
Here we summarize the construction in the special case where $I$ is a two-sided
ideal.

For a two-sided ideal $I$ of $\chrisA$ we define its auxiliary radical
\[
\sqrt[\alpha]{I} := \{ a \in \chrisA \colon -a^\ast a \in \Sigma_\chrisA+I_h\}.
\]
Clearly, $\sqrt[\alpha]{I}$ is a right ideal but
it may not
be a left ideal unless $\chrisA$ is centrally bounded. Therefore, we consider
the left ideal 
generated by  $\sqrt[\alpha]{I}$, i.e., 
\[
\sqrt[\beta]{I} := \chrisA \sqrt[\alpha]{I}
\]
which is always a two-sided ideal. By \cite[Proposition 5.7]{chmn}, we have the
following:

\begin{prop}
 \label{prop:Jaka-procedure}
If $I$ is a two-sided ideal of a $\ast$-algebra $\chrisA,$ then
\[
\sqrt[\beta]{I} \cup \sqrt[\beta]{\sqrt[\beta]{I}} \cup
\sqrt[\beta]{\sqrt[\beta]{\sqrt[\beta]{I}}} \cup \ldots =
\rr{I}.
\]
\end{prop}
%%-- end iffalse starts 30 lines above
\fi

\subsection{Proof of Proposition \ref{thm:mainfindim}}
% A proof of \ref{thm:mainfindim}
%  occupies the remainder of this  section.
%%This subsection proves  Theorem \ref{thm:mainfindim} (\ref{it:finGen}) and
%%(\ref{it:big}).
%%Before proving Theorem \ref{thm:mainfindim} (\ref{it:finGen}), we
 The length of the longest word in a noncommutative polynomial $f\in \FA$ is the
\df{degree} of $f$ and is denoted by $\deg( f)$. The set
of all words of degree at most $k$ is $\axs_k$, and $\FA_k$ is the vector
space of all noncommutative polynomials of degree at most $k$.

Given a subspace  $W$ of $\F\axs,$ let
\[ W_d = \{ w \in W \colon \deg(w) \leq d\}  = W\cap \FA_d,,\]
denote the elements of $W$ of degree at most $d$. 
Likewise, let
\[
W_d^{\hom} := \{w \in W \colon w = 0 \mbox{ or } w\ \text{is homogeneous of
degree}\ d \}, 
\]
 denote the homogeneous of degree $d$ elements of $W$.
  The \df{leading polynomial} of a nonzero polynomial $p \in \F\axs$ is
  the unique homogeneous polynomial $p^{\lead}$ such that $\deg(p) =
\deg(p^{\lead})$ and $\deg(p - p^{\lead}) < \deg(p)$. Let
\[W_d^{\lead} := \{0\} \cup \{ w' \in \F\axs_d^{\hom} \colon w' \mbox{ is the
leading polynomial of an element of } W \}.\]

\begin{proof}[Proof of Proposition \ref{thm:mainfindim} (\ref{it:finGen})]
If $\chrisA$ is a finitely-generated algebra, and $I \subset \chrisA$ is an
ideal such that the dimension of 
$\chrisA/I$ is finite, 
then $I$ is finitely generated as a left ideal by 
\cite[Lemma 3]{lewin}. 
The special case $\chrisA=\FA$ is all that is required for this proof.

Next, suppose that $I$ is finitely generated as a left ideal,
but $I \neq 0$.
There exists some degree $d$ such that $I$ is generated by
some
nonzero
polynomials of degree bounded by $d$.
By \cite{chmn}[Proposition 2.19] we can decompose $\F\axs$ as
\[\F\axs = I \oplus V\]
where
\[V = \F\axs V_d^{\hom} \oplus V_{d-1}\]
with
\[\F\axs_d^{\hom} = I_d^{\lead} \oplus V_d^{\hom}.\]
Further, by \cite{chmn}[Equation (2.3)] we have, for each $e \geq 0$,
\[
I_{d+e}^{\lead} = \F\axs_e^{\hom} I_d^{\lead}.
\]
Let $\iota \in I \setminus \{0\}$.
If $V_d^{\hom} \neq \{0\}$, then let $\nu \in V_d^{\hom} \setminus \{0\}$.
If $\deg(\iota) = e$, then the leading polynomial of $\iota \nu$ is in the space
\[\F\axs_e^{\hom} V_d^{\hom} \cap I_{d+e}^{\lead} = \F\axs_e^{\hom}(V_d^{\hom}
\cap I_d^{\lead}) = \{0\}.\]
 This equality leads to the contradiction $\iota \nu =0$.  Therefore
$V_{d}^{\hom} = \{0\}$.  Therefore
\[\dim(\F\axs/I) = \dim(V) = \dim(V_{d-1}) < \infty,\]
since $V_{d-1}$ is a subspace of $\F\axs_{d-1}$, which is finite dimensional.

\end{proof}

\begin{rem}
It is not true in general that if $\chrisA$ is a finitely-generated $*-$algebra and
$I$ is a two-sided ideal which is also a finitely-generated left ideal, then
$\dim(\chrisA/I) < \infty$. 
 As an example, consider the (commutative) polynomial
 ring  $\F [t_1, \ldots, t_g]$ in two or more variables $g$ with
  the trivial involution.  Since this algebra is commutative,
  there is no distinction between left and two-sided ideals.  In 
  particular, the ideal $I$ generated by $t_1$ is finitely generated, but
  $\F [t_1, \ldots, t_g]/I = \F [t_2,\ldots, t_g]$ is not finite dimensional.
%$\F$-algebra generated by $g\ge 2$ commuting variables $x_1, \ldots, x_g$.  Then
%the ideal $I$ generated by $x_1$ is finitely generated as a left ideal,
%but $\chrisA / I \cong \F [x_2, \ldots, x_g]$ is not finite dimensional.
\qed \end{rem}

We now turn our attention to proving Proposition \ref{thm:mainfindim}\eqref{it:big}.
Proposition \ref{structure} summarizes the structure theory of 
finite-dimensional formally-real $\ast$-algebras. Remark 
\ref{remstruct} explains the history of this result.

\begin{prop}
\label{structure}
If  $\chrisA$ is a finite-dimensional formally real $\ast$-algebra, then
$\chrisA$ is a 
$\ast$-algebra direct sum of formally real simple $\ast$-algebras. Moreover:
\begin{enumerate}
\item If $F=\CC,$ then every simple $\ast$-algebra is of the form
$\matof{n}{\CC}{}$
where $n \in \NN$ and the involution
is conjugate transpose.
\item If $F=\RR,$ then every simple $\ast$-algebra is of the form
$\matof{n}{\RR}{}$ or
$\matof{n}{\CC}{}$ or $\matof{n}{\mathbb H}{}$
where $n \in \NN$ and the involution is conjugate transpose.  Here $\mathbb{H}$
denotes the quaternions.
\end{enumerate}
In particular, $\chrisA$ has a faithful finite-dimensional
$\ast$-representation.
\end{prop}

\begin{proof}
Let $\chrisA$ be finite dimensional and formally real. By \cite[Theorem
2.2]{munn}, $\chrisA$ is semisimple.
By \cite[Chapter 0]{jac}, every semisimple algebra with involution is a
$\ast$-algebra direct sum of
simple $\ast$-algebras and every simple $\ast$-algebra is one of the following
types (where $D$ is a division algebra with involution): 
(i) $\matof{n}{D}{} \otimes {\matof{n}{D}{}}^{\mathrm{op}}$ with exchange
involution, (ii)
$\matof{n}{D}{}$ with conjugate transpose involution
or (iii) $\matof{2n}{D}{}$ with symplectic involution. The exchange and the
symplectic involution are clearly not formally real. (They are not even proper.)
Therefore every formally real simple $\ast$-algebra is of type (ii). Finally we
use the Frobenious theorem which 
says that $\RR$, $\CC$ and $\mathbb{H}$ are the only finite-dimensional division
algebras over $\RR$.
It remains to show that the only formally real involution on $\mathbb{H}$ is
the standard involution.
By the Noether-Skolem theorem, every involution $\#$ on $\mathbb{H}$ is of the
form $x^\#=h x^* h^{-1}$
where $\ast$ is the standard involution and $h \in D$. Now $x^{\#\#}=x$ implies
that $h^\ast h^{-1}$ is in $Z(\mathbb{H})=\RR$.
Since $h^{**}=h$, we get $h^\ast=\pm h$.
If $h^* =h$, then $h \in \RR$ and so $x^\#=h x^* h^{-1}=x^*$ for every $x \in
\mathbb{H}$.
If $h^\ast=-h$, then $h=\alpha i+\beta j+\gamma k$ for some $\alpha,\beta,\gamma
\in \RR$. It follows that 
$ihi+jhj+khk=(-\alpha i+\beta j+\gamma k)+(\alpha i-\beta j+\gamma k)+(\alpha i+\beta j-\gamma k)=\alpha i+\beta j+\gamma k =h$.
Multiplying through by $h^{-1}$ we get $1+ii^\#+jj^\#+kk^\#=0$.
Therefore, $\#$ is equal to $\ast$ in the first case while it is is not formally
real in the second case.
\end{proof}

\begin{rem}
\label{remstruct}
Part (1) of Proposition \ref{structure} is true even for proper involutions.
This is well-known (see e.g. \cite[Theorem 9.7.22]{palmer}) and it is also clear from our proof.

Part (2) is is a slight generalization of the structure theorem for positive involutions on finite-dimensional real $\ast$-algebras. 
The history of this result is explained in \cite[Section 2]{lewis} and \cite{ob}.
Recall that an involution $\ast$ on a finite-dimensional real algebra $\chrisA$ is \df{positive}
if $\operatorname{Tr}(a^\ast a)>0$ for every nonzero $a \in \chrisA$. 
Clearly, every positive involution is formally real but the converse is false; see \cite{cim}.

Part (2) does not generalize to proper involutions because there are proper involutions
on $\HH$ which are not standard. For example, a short computation shows that the involution 
defined by $i^\ast=i,j^\ast=j$ (and so $k^\ast=-k$) is proper but it is not formally real.
\qed
\end{rem}

We are now able to prove the following generalization of Proposition \ref{thm:mainfindim} \eqref{it:big}.

\begin{cor}
\label{cor:ihardx}
If $I$ is a $*$-ideal of a $\ast$-algebra $\chrisA$, then $I$ is real with $\dim(\chrisA/I) < \infty$
if and only if $I$ is the kernel of some finite-dimensional $\ast$-representation.
In particular, if $I$ is a $*$-ideal of a $\ast$-algebra $\chrisA$ such that
$\dim (\chrisA/\rr{I}) < \infty$, then $\rhard{I}=\rr{I}.$
\end{cor}

\begin{proof} First, if $I$ is a real $\ast$-ideal with $\dim(\chrisA/I) < \infty$,
then $\chrisA/ I$ is a finite-dimensional formally real $\ast$-algebra.
Therefore, Proposition \ref{structure} implies that $\chrisA/I$ has a
finite-dimensional faithful $\ast$-representation, which implies that $I =\rhard{I}$.

Conversely, if $I =\operatorname{ker} \pi$ for some $\pi \in \Pi$, then
$\chrisA/I \cong \operatorname{im}\pi$ implies  $\dim(\chrisA/I)  < \infty$.  
Further, if \[\sum_i^{\fin} p_i^*p_i \in \operatorname{ker} \pi,\]
then $\sum_i \pi(p_i)^*\pi(p_i) = 0$, which implies that each $\pi(p_i) = 0$, or
equivalently, each $p_i \in \operatorname{ker} \pi$.  Therefore $I$ is real.
\end{proof}

An alternative proof of Proposition \ref{thm:mainfindim} \eqref{it:big}
(which does not generalize to arbitrary $\ast$-algebras) will be given
in Section \ref{sec:leftZeroes}.

\section{No General Real Nullstellensatz for Free $\ast$-Algebras}
\label{sec:mainex}
  This cautionary section gives the details of the Toeplitz
  algebra mentioned at the outset of Subsection \ref{sec:infinitecodim}.
  It also contains a discussion of the Weyl algebra which has no bounded
  representations and  the proof of Proposition \ref{prop:jaka}.
  The additional theory of $\ast$-algebras used in the proof of 
  Proposition \ref{prop:jaka} is developed in Subsection \ref{sec:ostar}.
  These examples, taken together, support the general theme that 
  Nullstellensatz for various representation classes impose 
  serious restrictions on a $\ast$-algebra.

\begin{example}
\label{ex:toeplitz}
 Let $\chrisA=\F\langle x,x^* \rangle$ denote the free $\ast$-algebra in one
variable $x$ and
 let $I$ be the $\ast$-ideal of $\chrisA$ generated by $1-x^\ast x$.
 The algebra $\chrisA/I$ is called \df{the Toeplitz algebra}.
 Let $\mathbf{X}$ denote the shift operator  on $\ell^2(\NN)$ and let $\pi \colon
\chrisA \to L(\ell^2(\NN))$ denote the 
  (bounded) $\ast$-representation of  evaluation at $\mathbf{X}$.
 As is shown below,  $I=\operatorname{Ker} \pi$.
 It follows that $I$ is a real ideal and it satisfies the Real Nullstellensatz
over $\cC=\{\pi\}$.
 (Hence it also satisfies the Real Nullstellensatz over the class  of all bounded $\ast$-representations.)
 
 Clearly $\mathbf{X}^\ast \mathbf{X}=\mathbf{I}$, so $I \subseteq
\operatorname{Ker} \pi$. Conversely, let
 $p$ be an element of $\chrisA$ and let $\sum_{i=0}^m \sum_{j=0}^n c_{ij} x^i
(x^\ast)^j$
 be its canonical form modulo $I$.  Suppose that $\pi(p)=0$. It follows that for
every integer $k$,
 $$0=(\sum_{i=0}^m \sum_{j=0}^n c_{ij} X^i (X^\ast)^j)e_k=\sum_{j=0}^k
\sum_{i=0}^m c_{ij} e_{k-j+i}$$
 where $e_0,e_1,e_2,\ldots$ is the standard basis of $\ell^2(\NN)$. For $k=0$ we
get that $c_{i0}=0$ for every $i$.
 For $k=1$ we deduce that $c_{i1}=0$ for every $i$ and so on. Hence
$I=\operatorname{Ker} \pi$.
 
 On the other hand, $I$ does not satisfy the real Nullstellensatz over the class
$\Pi$ of all
 finite-dimensional $\ast$-representations (i.e.  evaluations on same size
square matrices over $F$.)
 Namely, we will show that the element $1-xx^\ast$ belongs to $\rhard{I}$
but it does not belong to $I$.
 We already know that $I=\rr{I}$.
 For a square matrix $\mathbf{Y}$, the relation $\mathbf{Y}^T \mathbf{Y}=\mathbf{I}$ is
equivalent to 
 $\mathbf{Y} \mathbf{Y}^T=\mathbf{I}$, hence $1-xx^\ast \in \rhard{I}$.
 Since $(1-\mathbf{X}\mathbf{X}^\ast)e_0=e_0 \ne 0$, it follows that $1-xx^\ast
\not\in I$. \qed
\end{example}

The following is also standard.

\begin{example}
\label{ex:weyl}
Let $I$ be the $\ast$-ideal in $\chrisA=\F\langle a,a^* \rangle$ generated by
$aa^\ast-a^\ast a-1$.
The algebra $\chrisA/I$ is called \df{the Weyl algebra}. It has a faithful $\ast$-representation, 
see \cite[Example2]{sch2}), but it does not have any bounded $\ast$-representations, see \cite[Theorem 13.6]{rud}.
Hence $I$ satisfies the Real Nullstellensatz over $\cC=\{\pi_0\}$ but it does
not satisfy the
Real Nullstellensatz over the class of all bounded $\ast$-representations of $\chrisA$.
\qed
\end{example}

For nice generalizations of Examples \ref{ex:toeplitz} and \ref{ex:weyl} see
\cite{ss} and \cite{sch1}.

 The more elaborate negative example of Proposition \ref{prop:jaka} requires
  some additional theory of $\ast$-algebras which is developed
  in Subsection \ref{sec:ostar}. The proof of the proposition follows
 in Subsection \ref{sub:proofJaka}.

\subsection{$O^*$-Representable Algebras}
\label{sec:ostar}
For a pre-Hilbert space $V$, the $\ast$-algebra $L(V)$ of all adjointable
linear operators on $V$ is formally real. 
It follows that the kernel of any $\ast$-representation is a real two-sided
ideal. 
In particular, every \df{$O^\ast$-representable} algebra
(i.e. a $\ast$-algebra having a faithful $\ast$-representation) is formally
real.

\begin{lem}
\label{ostar}
For a $\ast$-ideal $I$ of a $\ast$-algebra $\chrisA$, the following are
equivalent.
\begin{enumerate}
\item $I=\rchard{\cR}{I}$,
\item $I = \ichard{\cR}{T}$ for a subset $T$ of $\cR$,
\item $I$ is the kernel of some $\ast$-representation,
\item $\chrisA/I$ is $O^\ast$-representable.
\end{enumerate}
\end{lem}

\begin{proof}
Clearly, (1) implies (2). Conversely, (2) implies (1) because $\vchard{\cR}{\ichard{\cR}{T}} \supseteq T$ and
$\cI$ is inclusion-reversing. The implication
(3) implies (2) results from $\operatorname{Ker} \pi =\ichard{\cR}{\{\pi\}}$. To show that (2)
implies (3) take for $\pi$ 
the direct sum of all $\ast$-representations from $T$. The equivalence of (3)
and (4) is clear.
\end{proof}

Lemma \ref{ostar} will sometimes be used in combination with the following:

\begin{lem}
\label{hilba}
 Let $I$ be a $\ast$-ideal of  a $\ast$-algebra $\chrisA.$
  If there exists a  positive hermitian linear functional $L$ on $\chrisA$
   such that
\[
  I=\{a \in\chrisA \colon L(a^*a)=0\},
\]
 then the $\ast$-algebra $\chrisA/I$ is $O^\ast$-representable.
\end{lem}

\begin{proof}
Note that for a left ideal $I$ of $\chrisA$, the following are equivalent:
\begin{enumerate}
\item There exists a positive hermitian linear functional $L$ on $\chrisA$ such
that $I=\{a \in \chrisA \colon L(a^*a)=0\}$.
\item There exists an inner product $\langle \cdot,\cdot \rangle$ on the vector
space $\chrisA/I$ such that
$\langle [x y],[z] \rangle=\langle [y],[x^* z] \rangle$ for every $x,y,z \in
\chrisA$.
\item There exists an inner product on the $\ast$-algebra $\chrisA/I$ such that
the left regular representation of $\chrisA$ on the pre-Hilbert space
$\chrisA/I$ is a $\ast$-representation.
\end{enumerate}
Namely, to show that (1) implies (2), take $\langle [y],[z] \rangle := L(z^*y)$
and to show that (2) implies (1), take $L(x):=\langle [x],[1] \rangle$.
Clearly, (3) just rephrases (2). Finally (3) implies the claim because, 
 by Lemma \ref{lem:zi0}, 
the kernel of the left regular representation of $\chrisA$ on $\chrisA/I$ 
is equal to the largest two-sided ideal contained in $I$.
\end{proof}

Let $\cB=\cB(\chrisA)$ be the class of all bounded $\ast$-representations of
$\chrisA$. Note that a $\ast$-ideal
$I$ of $\chrisA$ satisfies $I=\rchard{\cB}{I}$ if and only if $R_\ast(\chrisA/I):=
\bigcap_{\pi \in \cB(\chrisA/I)} \operatorname{Ker} \pi=\{0\}$.

\subsection{Proof of Proposition \ref{prop:jaka}}
\label{sub:proofJaka}

\begin{proof}
In \cite{pop1}, it was shown that $\chrisA/I$ is formally real but not
$O^\ast$-representable. An alternative simpler 
 proof of this fact is given  below. Lemmas \ref{fr} and \ref{ostar}
now imply that
$$\rr{I}=I \ne \rchard{\cR}{I}.$$
In particular, $I$ is real. Since $\rchard{\cR}{I} \subseteq \rchard{\cC}{I}$ for
every representation class $\cC$, we have
$$\rr{I} \ne \rchard{\cC}{I}.$$
So $I$ does not satisfy the real Nullstellensatz over any representation class
$\cC$.
\end{proof}

\begin{prop}[\cite{pop1}]
  The $\ast$-algebra $\chrisA/I$ from Proposition \ref{prop:jaka} is formally real but is not $O^\ast$-representable. 
\end{prop}

\begin{proof}
To prove that $\chrisA/I$ is not $O^\ast$-representable we define a relation
$\le$ on $\chrisA/I$ by 
$u \le v$ if and only if  $v-u \in \Sigma_{\chrisA/I}$. 
By the second relation $aa^\ast \le 1$, hence $a^\ast a \le q$ by the
first relation. Suppose that $a^\ast a\le k$ for some $k \in \RR$.
Since $k(k-aa^\ast)=(k-aa^\ast)^2+a(k-a^\ast a)a^\ast$,
it follows that $a a^\ast  \le k$. Hence $a^\ast a\le q k$ by the first
relation. By induction $a^\ast a \le q^m$ for every $m$.
It follows that $a$ is in the kernel of every  $\ast$-representation of
$\chrisA/I$. 

To prove that $\chrisA/I$ is formally real we use the relations
$aa^\ast=1-xx^\ast$ and $a^\ast a=q(1-xx^\ast)$ to 
reduce each element of $\chrisA/I$ to its \df{canonical form}
whose monomials do not contain $a^\ast a$ or $aa^\ast$ as subwords.
Such monomials are linearly independent, so the canonical form is unique. 
Suppose that 
$$\sum_{i=1}^n p_i p_i^\ast=0$$
for some nonzero $p_1,\ldots,p_k \in \chrisA/I$.
Write each $p_i$ in its canonical form as 
$$p_i=\sum_{j=1}^{m_i} z_{ij} (a^\ast)^j+\sum_{k=1}^{m_i} w_{ik} a^k +t_i$$
where neither of the monomials in $z_{ij},w_{kj},t_i$ has either $a$ or $a^\ast$
as its last letter. Let $d$ be the
minimum of degrees of all monomials that appear in $z_{ij},w_{kj},t_i$. (Note
that degree is well-defined for canonical forms.)
Write $z_{ij}=z_{ij}'+z_{ij}''$, $w_{ik} = w_{ik}' + w_{ik}''$, and $t_{i} =
t_i' + t_i''$, where all monomials in $z_{ij}'$, $w_{ik}'$, and $t_{i}'$ have
degree $d$ while those in $z_{ij}''$, $w_{ik}''$, and $t_{i}''$ have degrees
$>d$.
Since
$$(a^\ast)^m a^m=q^m-\sum_{l=0}^{m-1}q^{m-l}(a^\ast)^l xx^\ast a^l 
\quad \text{and} \quad
a^m (a^\ast)^m=1-\sum_{l=0}^{m-1}a^l xx^\ast (a^\ast)^l$$
for each $m$, the canonical form of $\sum_{i=1}^n p_i p_i^\ast$ is equal to
$s+o$ where
\begin{equation}
\label{eq1}
s:=\sum_{i,j} q^j z_{ij}'(z_{ij}')^\ast+\sum_{i,k} w_{ik}' (w_{ik}')^\ast
+\sum_i t_i' (t_i')^\ast
\end{equation}
consists of monomials of degree $2d$ and $o$ consists of monomials of degree
$>2d$. By the uniqueness of the
canonical form, we have that 
\begin{equation}
\label{eq2}
s=0.
\end{equation}
Let us order the monomials that appear in $z_{ij}',w_{ik}',t_i'$
lexicographically and let $m$ be the first of them. Therefore, we can rewrite
\eqref{eq1} and \eqref{eq2} as
\begin{equation}
\label{eq3}
0=\sum_l (\alpha_l m+r_l)(\alpha_l m+r_l)^\ast
\end{equation}
with $m$ before each monomial of every $r_l$. Since 
$y_1:=\sum_l (\alpha_l m)(\alpha_l m+r_l)^\ast$ and $y_2:=\sum_l r_l(\alpha_l
m+r_l)^\ast$
have disjoint monomials and $y_1+y_2=0$ by \eqref{eq3}, it follows that
$y_1=y_2=0$. 
Canceling $m$ in $y_1^\ast=0$ gives $\sum_i \bar{\alpha}_l(\alpha_l m+r_l)=0$.
Consequently, $\sum_l \bar{\alpha}_l \alpha_l=0$, a contradiction with $p_i \ne
0$.
\end{proof}

On the positive side, every $\ast$-ideal generated by unshrinkable words (i.e.
words not decomposable as $d^\ast d u$
or $u d^\ast d$ where $d,u$ are words and $d$ is nonempty) is real and satisfies
the Real Nullstellensatz over $\cR$,
see \cite{chkmn,pop1,pop2}. It is not known yet if such ideals satisfy the Real
Nullstellensatz over the class $\cB$. 
Many more positive examples are given in Section 4 of \cite{pop1}.

\begin{comment}

\begin{rem}
It is tempting to replace $\ast$-representations over pre-Hilbert spaces with
$\ast$-representations over pre-Hilbert modules.
Then trivially every ideal satisfies the Real Nullstellensatz. Namely, let
$\lambda \colon \chrisA \to L(\chrisA/I)$ be
the left regular representation $\lambda(a)(x+I)=ax+I$. If $I$ is a real ideal,
then we have an $\chrisA/I$-valued
inner product on $\chrisA/I$ defined by $\langle x+I,y+I \rangle:=y^\ast x+I$.
Clearly, $\lambda$ is a $\ast$-representation
with respect to this inner product and $\operatorname{Ker} \lambda=I$. This
direction does not seem interesting to follow. 
\qed \end{rem}

\end{comment}

\section{Ideals Generated by Analytic Polynomials}
\label{sec:analytic}
  A classical commutative polynomial 
  of several complex variables is analytic
  if it depends only on $z=(z_1,\dots,z_g)$ 
  (and not on $\overline{z}$).
  By analogy, a polynomial in $\F\axs$ is \df{analytic} if it does not contain
  any  $x_j^*$ variables. For example, $ p=x_1 x_2 +x_1$
  is analytic and $p= x_1^* x_2 $ is not. 
  Let $\F\ax$ denote the analytic polynomials in $\FA$.

 Let $I(P)$ denote the $\ast$-ideal generated by a collection  $P$ (not necessarily finite)
  of analytic polynomials. 
  Recall that $I(P)$ has the \df{(hard) Nullstellensatz property} if
\[
  \rhard{I(P)}  = I(P).
\]

 The ideal $I_p=I(\{p\})$  generated by the single polynomial 
  $p = x_1x_2 - x_2x_1 + 1$ does not have the  hard
  Nullstellensatz property.  Indeed, $p(X)$ is never equal to $0$
 since $\operatorname{Tr}(p(X)) = \operatorname{Tr}(1) >0 $.  
   Thus $\rhard{I_p} = \F\axs.$   However,   $I_p\ne\F\axs$ 
  (the Gr\"obner basis for $I$ is $\{p\}$, and so $1 \not\in I_p$). 
%Hence
%   $1\in \rhard{I_p} \cap \F\ax$ but $1\notin I_p\cap\F\ax$.
  On the other hand, we do not know
 of a polynomial $p$ for which $p(X) =0$ has a solution
 but $\rhard{I_p} \neq I_p$.

In this section we give conditions on $P$ which imply that
  $I(P)$ has the hard Nullstellensatz property.
 For instance,  
 Theorem \ref{thm:analHardNSS} says if $I(P)$ is 
homogeneous, then it has the Nullstellensatz property.  
  We note that left ideals with analytic generators have a good
\nss\ (see \cite{HMP07}) using zeros of the type as described
  in Section \ref{sec:leftZeroes}.

We start by introducing Gr\"obner basis machinery which is required to prove
Theorem \ref{thm:analHardNSS}.

\subsection{Non-Commutative Gr\"obner Bases}
A \df{monomial order} $\prec$ on $\axs$ is a total order
on the elements of $\axs$ with the following properties.
\begin{enumerate}
 \item $\prec$ is a well ordering, that is, each nonempty subset of $\axs$ has a
minimal element.
\item If $a, b \in \axs$ with $a \prec b$, and $c \in \axs$, then $ca \prec cb$
and $ac \prec bc$.
\end{enumerate}
Given a monomial order $\prec$, every nonzero polynomial $p \in \F\axs$
can be written as $p=\sum_{i=1}^s c_i t_i$ where $c_1,\ldots,c_s$ are nonzero
elements of $\F$
and $t_1 \succ \ldots \succ t_s$ belong to $\axs$. In this case, $\lm{p}:= t_1$ is
the \df{leading monomial} of $p$, according to $\prec$, and $\lc{p}:= c_1$
is the \df{leading coefficient} of $p$, according to $\prec$.
We say that $p$ is \df{monic} if $\lc{p} =1$.
Given two words $a, b \in \axs$, we say $a$ \df{divides} $b$ if
$b = cad$ for some $c, d \in \axs$.

Given a two-sided ideal $I \subset \F\axs$, a \df{reduced Gr\"obner basis}
for
$I$ is a set $G \subset I$ with the following properties:
\begin{enumerate}
 \item for each $f \in I$ there exists $g \in G$ such that $\lm{g}$ divides
$\lm{f}$.
 \item each element of $G$ is monic, and
 \item if $g_1$ and $g_2$ are distinct elements of
$G$, then $\lm{g_1}$ does
not divide any term of $g_2$.
\end{enumerate}

By \cite[Proposition 1.1]{tm}, every two-sided ideal of $\F \axs$ has a unique
reduced Gr\" obner basis.

 Let $\ax \subset \axs$ be the set of \df{analytic monomials} in $\axs$,
and let $\F\ax \subset \F\axs$ be the set of \df{analytic polynomials} in
$\F\axs$.

\begin{prop}
\label{prop:GBofAnStarIdeal}
Let $\prec$ is a monomial order on $\F \axs$.
 Let $I \subset \F\axs$ be a $\ast$-ideal generated by some
nonzero analytic polynomials.  Then the reduced Gr\"obner basis for $I$ is of
the form $G \cup H^*$, where $G$ and $H$ consist entirely of analytic
polynomials.
\end{prop}

\begin{proof}
 Let $G$ be the reduced Gr\"obner basis for $I \cap \F\ax$ and let $H^*$ be the
reduced Gr\"obner basis for $I^* \cap \F\ax^*$. Since $I$ is
generated by some analytic polynomials as a $\ast$-ideal, it is generated by
some analytic and some antianalytic
polynomials as a two-sided ideal. Therefore $G\cup H^*$ generates $I$ as a
two-sided ideal in $\F\axs$.
It is clear that $G \cup H^*$ satisfies conditions (2) and (3) of being a
reduced Gr\"obner basis for $I$.
To prove that $G \cup H^*$ satisfies condition (1) we need the following:

\vspace{.5 em}

\noindent \textit{Claim 1. \ 
Every element of $I$ is a linear combination of polynomials of the form
\begin{equation}
 \label{eq:conditionsPi}
p_1p_2 \cdots p_k,
\end{equation}
where $p_i$ alternate between analytic and antianalytic polynomials and each
$p_i$ is either
\begin{enumerate}[(a)]
\item  a nonconstant monomial which is not the leading monomial of an
element of $I;$ or 
\item  a polynomial $asb$, where $a,b \in \axs$ and $s \in G \cup H^*$.
\end{enumerate}
Moreover, we can assume that at least one $p_i$ in each product is of the second
form.
}
\medskip

Since $G\cup H^*$ generates $I$ as a two-sided ideal, each $f \in I$ is of the
form 
\[f=\sum_{j, \text{finite}} a_j g_j b_j+\sum_{j', \text{finite}} c_{j'}
h_{j'}^*d_{j'}\]
where $g_j \in G$, $h_{j'} \in H$ and $a_j,b_j,c_{j'},d_{j'} \in \F \axs$.
Now, every element of $\F \axs$ is clearly a linear combination of products $r_1
\cdots r_m$
where $r_i$ alternate between analytic and antianalytic monomials.
If $r_i$ is analytic then, by \cite[Theorem 1.1, assertion 1]{tm}, we have that 
$r_i=\iota_i+\omega_i$ for some $\iota_i \in I \cap \F \ax$ and some $\omega_i$
which is a linear combination of analytic monomials that are not the leading
monomial of some element of $I \cap \F \ax$.
Further, $\iota_i=\sum_{i', \text{ finite}} a_{i'} g_{i'} b_{i'}$ for some
$g_{i'} \in G$ and $a_{i'},b_{i'} \in \F \ax$.
A similar argument prevails if $r_i$ is antianalytic thus
 completing the proof of Claim 1. 
%Claim 2 below clearly follows from Claim 1.

\vspace{.5 em}

\noindent
\textit{Claim 2. \ 
If $p_1 \cdots p_k$ and $q_1 \cdots q_{\ell}$ are of the form
\eqref{eq:conditionsPi} and
\[\lm{p_1 \cdots p_k} = \lm{q_1 \cdots q_{\ell}},\]
then $q_1 \cdots q_k -  p_1 \cdots p_k$ is a linear combination of polynomials
of the form
\eqref{eq:conditionsPi}.
}
\medskip

The assumption implies that $\lm{p_1} \cdots \lm{p_k} = \lm{q_1} \cdots
\lm{q_{\ell}}$.
It follows that $k = \ell$ and $T(p_i)=T(q_i)$ for each $i$. Moreover, for each
$i$:
\begin{enumerate}[(a)]
\item either $\lm{p_i} = p_i = q_i;$ or 
\item $p_i$ and $q_i$ are both analytic and in the ideal generated by $G;$ or
\item $p_i$ and $q_i$ are both antianalytic and in the ideal generated by
$H^*$.
\end{enumerate}
Consequently, 
$q_1 \cdots q_k -  p_1 \cdots p_k = q_1 \cdots q_k - (p_1 - q_1 + q_1) \cdots
(p_k - q_k + q_k) = \bar{q}$
where $\lm{\bar{q}} \prec \lm{q_1 \cdots q_k}$
and $\bar{q}$ is a linear combination of polynomials of the form
\eqref{eq:conditionsPi},  proving  Claim 2.

\medskip

To complete the proof of the proposition, let $f \in I$ be given.
 By Claim 1, 
\[
f=\sum_{k=1}^n c_k z_k
\]
where $c_k \in \F$ and $z_k$ are of the form \eqref{eq:conditionsPi}.
It can be  assumed that there is an  $m$ such that
  $\lm{z_1} = \ldots = \lm{z_m} \succ
\lm{z_{m+1}},\ldots,\lm{z_n}$.
 For each $k=2,\ldots, m,$ Claim 2 implies  $z_k-z_1$ is a linear combination 
of polynomials of the form \eqref{eq:conditionsPi} and $\lm{z_k-z_1} \prec
\lm{z_1}$.
It follows that
\[
f=(\sum_{k=1}^m c_k) z_1+\bar{z}
\]
where $\lm{\bar{z}} \prec \lm{z_1}$.
Now, if $\sum_{k=1}^m c_k \ne 0$, then $\lm{f}=\lm{z_1}$ is divisible by 
the leading coefficient of an element of $G \cup H^*$. If $\sum_{k=1}^m c_k =
0$,
we continue by induction. Therefore the leading monomial of every element of $I$
is
in the ideal generated by the leading monomials of elements of $G \cup H^*$.
\end{proof}

\subsection{$\ast$  Ideals with Analytic Generators}
 A \df{graded order} $\prec$ on $\axs$ is a monomial order
such that $a \prec b$ if $\deg a < \deg b$. (Recall that $\deg a$ is the number
of letters in the word $a$).

\begin{lem}
\label{lem:IplusW}
  Let $I \subset \F\axs$ be a two-sided ideal.
 If  $\prec$ be a graded order on $\axs$ and if
  $W$ is the space spanned by all monomials which are not the leading
monomial of an element of $I$,   then $\F\axs = I \oplus W$.  Further, if
$p = \iota + \omega \in \F\axs$, with $\iota \in I$ and $\omega
\in W$, then $\deg(p) = \max\{\deg(\iota), \deg(\omega)\}$.
\end{lem}

\begin{proof}
The first part is true for every monomial order; see assertion (1) in
\cite[Theorem 1.1]{tm}.

Next, suppose $p = \iota + \omega$, where $\iota \in I$ and $\omega \in W$.
The only way that $\deg(p) \neq
\max\{\deg(\iota), \deg(\omega)\}$ is if the highest degree terms of $\iota$
and $\omega$ cancel each other out.
In this case $\lm{\omega}=\lm{-\iota}$ since $\prec$ is a graded monomial order.
On the other hand, $\lm{\omega}$ is not the leading monomial of an element from
$I$; a contradiction.
\end{proof}

Proposition \ref{prop:analLinFun} and Corollary \ref{cor:genByAnalReal2} are the main results of this subsection.

\begin{prop}
\label{prop:analLinFun}
 Let $I \subset \F\axs$ be a $\ast$-ideal generated by
analytic polynomials. 
There exists a positive hermitian linear functional $L$ such that
\begin{equation*}
I = \{ a \in \F\axs \colon L(a^*a) =0 \}.
\end{equation*}
Hence, the $\ast$-algebra $\F\axs/I$ is $O^*$-representable.
\end{prop}

\begin{cor}
\label{cor:genByAnalReal2}
 If  $I \subset \F\axs$ is a $\ast$-ideal generated by
analytic polynomials, then 
$$\rr{I} =\rchard{\cR}{I}= I.$$
\end{cor}

\begin{proof}[Proof of Corollary \ref{cor:genByAnalReal2}]
Lemma \ref{ostar} and 
Proposition \ref{prop:analLinFun} imply 
 $\rchard{\cR}{I} = I$. 
 By Proposition \ref{prop:basicin} we have  $\rr{I} \subset \rchard{\cR}{I}$
 and of course $I \subset \rr{I} $, so we get the conclusion.
\end{proof}

\begin{proof}[Proof of Proposition \ref{prop:analLinFun}]
Let $\prec$ be a graded 
%$\ast$-lexicographic 
monomial order on $\axs$ so that $a \prec b$ if $\deg a < \deg b$.
Let $W$ be the space spanned by all monomials which are not the leading
monomial of an element of $I$ so that, by Lemma \ref{lem:IplusW}, $\F\axs = I
\oplus W$.
We will construct a positive hermitian linear functional $L$ on $\F\axs$ such
that
\begin{equation}
\label{eq:Lprop}
I = \{ a \in \F\axs \colon L(a^*a) = 0\}
\end{equation}
as follows.
Set $\tilde{L}(I) = \{0\}$.  For each monomial $m \in W$, set 
\[\tilde{L}(m) = \left\{ \begin{array}{cl} 0 & \text{ if } m \text{ is not a
square} \\
c_d & \text{ if } m \text{ is a square of degree } 2d \end{array} \right.\]
where each $c_d > 0$ is a constant to be chosen inductively.
Finally, we define $L(a)$ to be
\[L(a) := \frac{1}{2} \tilde{L}(a) + \frac{1}{2} \tilde{L}(a^*)^*\]
so that $L$ is hermitian.  Note that if $\iota \in I$, then $\iota^* \in I$, so
$L(\iota) = 0$.

We will need the following:

\vspace{.5 em}

\textit{Claim. \ Let $m_1, m_2 \in W$ be monomials of degree $\le
d$. If $m_1 \ne m_2,$
then $L(m_1^*m_2)$ depends only on $c_1,\ldots,c_{d-1}$.}

\medskip

It suffices to show that $\tilde{L}(m_1^*m_2)$ depends only on
$c_1,\ldots,c_{d-1}$.

If $m_1^* m_2 \in W$, either $\tilde{L}(m_1^*m_2)=0$ or
$\tilde{L}(m_1^*m_2)=c_e$ for some $e<d$
because $m_1 \ne m_2$ implies that $m_1^*m_2$ is not a square of degree $2d$.

If $m_1^*m_2 \not\in W$, we can
decompose $m_1^*m_2$ as $\iota + \omega$, where $\iota \in I$,
$\lm{\iota} = m_1^*m_2$, and $\omega \in W$, and
$\deg(\iota), \deg(\omega) \leq \deg(m_1^*m_2)$.

If $\deg(m_1^*m_2) < 2d$, then $\omega$ is spanned by monomials $u$ which either
are
not squares, in which case $\tilde{L}(u) = 0$, or which are squares, in which
case
$\tilde{L}(u)$ depends only on $c_0, \ldots, c_{d-1}$ because $\deg u < 2d$.

If $\deg(m_1^*m_2) = 2d$, then $\deg m_1 = \deg m_2 = d$.  Let $m_1$ and $m_2$
be of the
form
\[m_1 = u_1 \cdots u_k \quad \mbox{and} \quad m_2 = v_1 \cdots v_{\ell},\]
where the $u_i$ and $v_j$ alternate between being nonempty analytic and
antianalytic
words. Since $m_1,m_2 \in W$, each $u_i$ and each $v_j$ is not the leading
monomial of an element
of $I$. On the other hand, since $m_1^*m_2 \not\in W$,
\[m_1^*m_2 = u_k^* \cdots u_1^*v_1 \cdots v_{\ell}\]
is the leading monomial of some $p \in I$.
Let $G \cup H^*$ be the reduced Gr\"obner basis
of $I$ given by Proposition \ref{prop:GBofAnStarIdeal}.
By property (1) of reduced Gr\"obner bases, $u_k^* \cdots u_1^*v_1 \cdots
v_{\ell}=\lm{p}$
is divisible by the leading monomial of some $q \in G \cup H^*$.
Note that $\lm{q}$ is either analytic or antianalytic but it divides neither of
the words $u_i$ and $v_j$.
The only way this can happen is that $\lm{q}$ divides $u_1^*v_1$ and so
$u_1^*v_1$ is either analytic or antianalytic.
Let us decompose $u_1^*v_1=\iota_1 + \omega_1$, where $\iota_1 \in I$,
$\omega_1 \in W$, and both $\iota_1$ and $\omega_1$ are (anti)analytic
if $u_1^*v_1$ is (anti)analytic.  Also, by Lemma \ref{lem:IplusW},
$\deg(\iota_1), \deg(\omega_1) \leq \deg(u_1^*v_1)$.  Therefore
\[\tilde{L}(m_1^*m_2) = \tilde{L}(u_k^* \cdots u_2^*\iota_1 v_2 \cdots v_{\ell})
+ \tilde{L}(u_k^* \cdots u_2^*\omega_1 v_2 \cdots v_{\ell}).\]
We have $\tilde{L}(u_k^* \cdots u_2^*\iota_1 v_2 \cdots v_{\ell}) = 0$ since
$\tilde{L}(I) =
\{0\}$.  Next, the degree $2d$ terms of $u_k^* \cdots u_2^*\omega_1 v_2
\cdots v_{\ell}$ cannot be squares since  the middle two letters of each degree
$2d$ word of come from pieces of terms of $\omega_1$, which is either analytic
or
antianalytic---the middle piece of a square word is always of the form $yy^*$,
where $y$ is a letter.  Therefore $\tilde{L}(u_k^* \cdots u_2^*\omega_1 v_2
\cdots v_{\ell})$
does not depend on $c_{d}$ since $u_k^* \cdots u_2^*\omega_1 v_2 \cdots
v_{\ell}$ 
has no squares of degree $2d$ in it.
This completes the proof of the claim.

\medskip

Let $M_d$ be a vector whose entries are all monomials of degree $d$ in $W$.
Consider $A = L(M_d^*M_d)$, which is defined by evaluating the entries of $M_d^*M_d$ with the functional $L$.
First, since $L$ is hermitian, clearly $A$ is as well.
Each monomial $m_1^*m_2$, with $\deg m_1 = \deg m_2 = d$, is distinct.  If $m_1
\neq
m_2$, then, by the Claim, $L(m_1^*m_2)$ does not depend on $c_d$.
Finally, if $m_1 =m_2$, then $L(m_1^*m_1) = c_d$.  Therefore the matrix $A$ is
of the form
\[A = c_d \operatorname{Id} + F_d \]
$\operatorname{Id}$ is an identity matrix of
appropriate size, and where $F_d$ is hermitian and depends only on
$c_0, \ldots, c_{d-1}$.  Further, by the Claim, the value of $L$ on
$\F\axs_{2d-1}$ depends only on $c_0, \ldots, c_{d-1}$.  Therefore
Lemma \ref{lem:overlap}, below,  gives the result.
\end{proof}

The following technical lemma was used at the end of the proof of
Proposition \ref{prop:analLinFun}.
It will be also used in the proof of Proposition \ref{prop:homRealOp}.
If $A = (a_{ij})_{1 \leq i,j \leq n}$ is a matrix of polynomials, and $L$ is a linear functional on $\F\axs$, 
let $L(A)$ denote the matrix $(L(a_{ij}))_{1 \leq i,j \leq n}$.

\begin{lem}
\label{lem:overlap}
Let $I \subset \F\axs$ be a left ideal.
Fix a graded order $\prec$, and let $W$ be the space spanned by all monomials
which are not the leading monomial of an element of $I$ so that $\F\axs = I
\oplus W$.
For each degree $d$, let $M_d$ be the row vector whose entries
are all monomials of degree $d$ in $W$.

Suppose
there exist positive definite matrices $A_0, \ldots, A_d, \ldots$ such that
for any positive constants $c_0, \ldots, c_d, \ldots$, a
well-defined linear functional $L$ on $\F\axs$ can be defined inductively with
the following properties:
\begin{enumerate}
 \item $L(I + I^*) = \{0\}$.
 \item If $\deg(p) < 2d$ for some
$d$, then the definition of $L(p)$ depends only on the choice of $c_0,
\ldots, c_{d-1}$.
 \item $L(M_d^*M_d) = c_dA_{d} + F_d$, where the hermitian matrix $F_d$ depends
only on the choice of $c_0, \ldots, c_{d-1}$.
\end{enumerate}
Then there exist values of $c_0, \ldots, c_d, \ldots$ such that the linear
functional so defined satisfies $L(a^*a) \geq 0$ for each $a \in \F\axs$ and
equals $0$ if and only if $a \in I$.
\end{lem}

\begin{proof}
If $1 \in I$, then the problem is trivial.  Otherwise, for $d = 0$, we
must have $L(1) = c_0 A_{0} + F_0$, where $A_{0}$ is a positive scalar.
Therefore, for a sufficiently large value of $c_0$ we get $L(1) > 0$.

Next assume inductively
assume inductively that $c_0, \ldots, c_{d-1}$ are defined so that
$L(b^*b) > 0$ for each $b \in \F\axs_{d-1} \setminus I_{d-1}$.
Let $a \in \F\axs_d$ so that $a$ can be decomposed as
$a = \iota + \omega$, where $\iota \in I$ and $\omega \in W$.
By Lemma \ref{lem:IplusW}, $\deg(\iota), \deg(\omega) \leq d$.
Since $L(I + I^*) = \{0\}$,
\[L(a^*a) = L(\iota^* \iota)+ L(\iota^*a) + L(\omega^*\iota) + L(\omega^*\omega)
=L(\omega^*\omega).\]
Further, suppose $a \not\in I$, which implies that $\omega \neq 0$.

Let $N_{d-1}$ be the row vector whose entries are all words in $W$ of length
less than $d$. Then $\omega = M_d
\alpha_d + N_{d-1}\alpha_{d-1}$ for some constant column vectors $\alpha_d,
\alpha_{d-1}$.  We see that
\begin{equation}
\label{eq:matrixFormLGamma}
L(\omega^*\omega)
= \begin{pmatrix}
\alpha_d\\ \alpha_{d-1}
\end{pmatrix}^*
\begin{pmatrix}
A&B\\ B^*&C
\end{pmatrix}
\begin{pmatrix}
\alpha_d\\ \alpha_{d-1}
\end{pmatrix},
\end{equation}
where
\[A = L(M_d^* M_d), \quad B = L(M_d^* N_{d-1}) \quad \mbox{and} \quad C =
L(N_{d-1}^*N_{d-1}). \]
If $\alpha_{d} = 0$, then $\deg(\omega) < d$, so $L(\omega^*\omega) =
\alpha_{d-1}^* C \alpha_{d-1} > 0$ since $\omega \not\in I$.  Since $L$ is
hermitian, clearly $C$ is also hermitian.  Since $\alpha_{d-1}$ is arbitrary,
this
implies that $C$ is positive definite.
Next, $B$ depends on polynomials of degree less than $2d$, so by assumption $B$
depends only
on $c_0, \ldots, c_{d-1}$, which are already determined.  Next consider $A =
c_d A_{d} + F_d$.
 We see that $c_0, \ldots, c_{d-1}$ are already determined, and since
$A_{d} \succ 0$, we can choose $c_d$ sufficiently large so
that the matrix
\[
\begin{pmatrix}
c_dA_d + F_d &B\\ B^*&C
\end{pmatrix}
\]
is positive definite.  Given (\ref{eq:matrixFormLGamma}), this
implies that $L(\omega^*\omega) > 0$. The result therefore follows by induction.
\end{proof}

\subsection{Homogeneous Analytic Ideals}
\label{sec:analyticMain}

 An (two-sided, left, right) ideal $I \subset \F\ax$ 
 is called \df{homogeneous} if it is
 generated by homogeneous polynomials, not necessarily of the
same degree.

\begin{prop}
\label{prop:homRealOp}
 If $I \subset \F\axs$ is a real, homogeneous left ideal (not necessarily
 finitely generated), then  there exists a
positive hermitian $\F$-linear functional $L$ on $\F\axs$
\[I = \{ \iota \colon L(\iota^*\iota) = 0\}.\]
\end{prop}

\begin{proof}

By the proof of \cite[Theorem 4.1]{chmn}, for each degree $d$ there exists a
positive hermitian $\F$-linear
functional $L_{2d}$ on $\F\axs_{2d}^{\hom}$ such that
\[I_d^{\hom} = \{ \iota \in \F\axs_{d}^{\hom} \colon L_{2d}(\iota^*\iota) =
0\},\]
and such that $L_{2d}(\iota) = 0$ for each $\iota \in (I + I^*)_{2d}^{\hom}$.
Define the linear functional $L$ on $\F\axs$ to be $0$ on odd degree monomials
and to be $c_d L_{2d}$ on
$\F\axs_{2d}^{\hom}$, where each $c_d$ is a positive constant to be chosen.
Note that since $I$ is homogeneous, by construction $L(I + I^*) = \{0\}$ since
each homogeneous  polynomial in $I$ is mapped to $0$.
Also, clearly $L$ is hermitian.

Consider $A = L(M_d^*M_d)$.
First, since $L$ is hermitian, clearly $A$ is as well.
Next, if $M_d \alpha  \in \F\axs_d^{\hom}$ for some constant column vector
$\alpha \neq 0$, then by linearity
\[
 L(\alpha^*M_d^*M_d \alpha) =  c_d \alpha^*L_d(M_d^*M_d) \alpha
\]
which is positive by assumption.
Therefore $L_d(M_d^*M_d) \succ 0$ and $A = c_d L_d(M_d^*M_d)$. Further, the
definition of $L$ on $\F\axs_{2d-1}$ depends only on $c_0, \ldots, c_{d-1}$.
 An application of 
Lemma \ref{lem:overlap} gives the
result.
\end{proof}

\begin{prop}
 \label{prop:homAnalHard}
  If  $I \subset \RR\axs$ is a $\ast$-ideal generated by
  homogeneous analytic polynomials, then 
  for each degree $d$, there exists a tuple of matrices $X$ 
  such that $\iota(X) =0$ for each $\iota \in I$ and $q(X) \neq 0$ for each $q \not \in I$
   with degree at most $d$. 
\end{prop}

\begin{proof}
 Fix $d \in \NN$.
Let $I^{(d)}$ be the $\ast$-ideal generated by $I$ as well as by all analytic
monomials of degree $d+1$.
In this case, $(I^{(d)})_d = I_d$.
By Proposition \ref{prop:analLinFun} there exists a nonnegative hermitian linear
functional $L_d$ such that
\[I^{(d)} = \{ a \in \F\axs \colon L_d(a^*a) = 0\}.\]
We follow the GNS construction to define $\cH$ to be the pre-Hilbert space
defined as the 
vector space $\F\axs / I^{(d)}$ with inner product
\[\langle [a], [b] \rangle := L_d(b^*a),\]
and a tuple of linear operators $\widetilde{X}=(\tilde{X}_1,\ldots,\tilde{X}_g)$
on $\cH$ such
that 
$\tilde{X}_i [r] = [x_i r]$ for each $i=1,\ldots,g$ and each $r \in \F\axs$.
Clearly, $\tilde{X}_i^* [r] = [x_i^* r]$ for each $i=1,\ldots,g$ and for each $r
\in
\F\axs$, which implies that $q(\widetilde{X})[r] = [qr]$ for each $q,r
\in \F\axs$.

Define $\cW \subset \cH$ to be the space
\[\cW = \{ [ab] \colon a,b \in \F\axs,\ a \mbox{ analytic, } \deg(b) \leq
d\}.\]
Since every analytic monomial of degree greater than $d$ is in $I^{(d)}$, the
space $\cW$
is finite dimensional.  Let $X$ be the tuple of operators on $\cW$ defined by
\[X = (P_{\cW}\tilde{X}_1P_{\cW}, \ldots, P_{\cW} \tilde{X}_g P_{\cW}),\]
where $P_{\cW}$ is the self-adjoint projection map onto $\cW$.
If $a$ is analytic and $\deg(b) \leq d$, then
$P_{\cW} \tilde{X}_i P_{\cW} [ab] =P_{\cW} \tilde{X}_i  [ab]=P_{\cW}[x_i ab]
=[x_i ab]$ for each $i=1,\ldots,g$ and hence, 
\[\vartheta(X)[ab] = [\vartheta ab]\]
for each analytic $\vartheta$.
If $\iota \in I$ is one of the analytic generators of $I$, this implies that
\[\iota(X)[ab] = [\iota ab] = 0.\]
Therefore $p(X) = 0$ for each $p \in I$.
Also, if $\deg(q) \leq d$, then
\[q(X)[1] = [q].\]
It is clear that $q \in I^{(d)}$ if and only if $q \in I$.
Therefore if $q \not\in I$, then $q(X) \neq 0$.
\end{proof}

\begin{thm}
\label{thm:analHardNSS}
 If  $I \subset \RR\axs$ is a $\ast$-ideal generated by  homogeneous
  analytic polynomials, then  $\rhard{I} = I$.
\end{thm}

\begin{proof}
This follows directly from Proposition \ref{prop:homAnalHard}.
\end{proof}

\begin{proof}[Proof of Theorem \ref{thm:homAnalHilb}]
One can construct tuples of matrices $X^{(d)}$ on finite-dimensional Hilbert
spaces $\cH^{(d)}$ by Proposition
\ref{prop:homAnalHard} such that $\iota(X^{(d)}) = 0$
for each $\iota \in I$ and $p(X^{(d)}) \neq 0$ if $\deg(p) \leq d$ and $p
\not\in I$.  Since $I$ is homogeneous, one can scale each $X^{(d)}$ by a
scalar and still preserve $\iota(X^{(d)}) = 0$
for each $\iota \in I$ and $p(X^{(d)}) \neq 0$ if $\deg(p) \leq d$ and $p
\not\in I$.  Therefore choose each $X^{(d)}$ to have norm bounded by $1$.
Let $X := \bigoplus_{d \in \NN} X^{(d)}$ be an operator on $\cH := \bigoplus_{d
\in \NN} \cH^{(d)}$.  Then clearly $\|X\| \leq 1$ and $p(X) = 0$ if and only if
$p \in I$.
\end{proof}

We end this section with two remarks.

\begin{rem}
 Returning to the example at the outset of this section of
  the $\ast$-ideal $I_p$ of $\F\axs$ generated by $p=1+x_1x_2-x_2x_1$,
  note that it does not satisfy the condition 
\[
  \rhard{I(P)} \cap \F\ax = I(P)\cap \F\ax
\]
 which is, at least formally, weaker than the 
 hard Nullstellensatz property.
\qed \end{rem}

\begin{rem}
  The real radical of the two-sided ideal in $\F\axs$ generated by
  a collection of analytic polynomials $P$ is the 
  the $\ast$-ideal generated by $P$. 
\qed \end{rem} 

\section{Left Zeroes}
\label{sec:leftZeroes}

  There is a theory of Nullstellensatz for a left ideal $I$
  in a $*$-algebra $\chrisA$ (see \cite{chmn,chkmn} for the most 
  recent results and for historical references)
  and this section briefly explores the connections between Nullstellensatz
  for left ideals and the Nullstellensatz  in this article for two-sided ideals. 
  
  The main result
    of  this section is that, 
    for the important representation classes, 
    the  left  radical of a two-sided ideal 
    coincides with its hard radical. The machinery
    developed in this section leads to an alternate
    proof of Proposition \ref{thm:mainfindim} \eqref{it:big}.
    This machinery will also be used in Section \ref{sec:softZeroes}
    where the relations between hard and soft zeros (defined later)
   are established. 
   
%\subsection{Definitions}  

 Given a representation  class  $\cC$ of the $\ast$-algebra $\chrisA$, let
\[
   \cC_\lft=\{(\pi,v) \colon \pi \in \cC, v \in V_\pi\}.
\]  
The elements of $\cC_\lft$ will be considered as ``left real points'' of $\chrisA$.
We say that an element $(\pi,v)$ of $\cC_\lft$ is a \df{left zero} of an element $a \in \chrisA$ if $\pi(a)v=0$. 
If $T\subset \cC_\lft$, then 
\[
 \icleft{\cC}{T}=\{a \in \chrisA \colon \pi(a)v=0 \text{ for all } (\pi,v) \in T\} 
\]
is a left ideal in $\chrisA$ --- the \df{left vanishing ideal} of $T$.
In the case that $T$ is a singleton   $\{(\pi,v)\},$ it is convenient to 
abbreviate  $\icleft{\cC}{\{(\pi,v)\}}$ to $\icleft{\cC}{\pi,v}$.
\i{$\icleft{\cC}{\{(\pi,v)\}}$ to $\icleft{\cC}{\pi,v}$}
Given  $\cS \subseteq \chrisA$, let
\[ 
\vcleft{\cC}{\cS}=\{(\pi,v) \in \cC_\lft \colon \pi(s)v=0 \text{ for all } s \in S\} 
\]
\i{$\vcleft{\cC}{\cS}$}
be its \df{left variety} and let
\[ 
\rcleft{\cC}{\cS} :=  \icleft{\cC}{\vcleft{\cC}{\cS}} 
\]
\i{$\rcleft{\cC}{\cS} $}
be its \df{left radical}. If $J(\cS)$ is the left ideal generated by $\cS$, then clearly
\[
\vcleft{\cC}{\cS}=\vcleft{\cC}{J(\cS)} \quad \mbox{and} \quad  \rcleft{\cC}{\cS}=\rcleft{\cC}{J(\cS)}.
\]\i{$\vcleft{\cC}{\cS}$}
When $T \subseteq \Pi$ and $\cS \subseteq \chrisA$ and $\cC=\Pi$ (recall that $\Pi$ is the class of all finite-dimensional
$\ast$-representations), we will use the abbreviations
\[
\ileft{T}=\icleft{\Pi}{T}, \quad \vleft{\cS}=\vcleft{\Pi}{\cS} \quad \mbox{and} \quad \rleft{\cS}=\rcleft{\Pi}{\cS}.
\]
\i{$\ileft{T}=\icleft{\Pi}{T}$}
\i{$\vleft{\cS}=\vcleft{\Pi}{\cS} $}
\i{$ \rleft{\cS}=\rcleft{\Pi}{\cS}$}

\subsection{Real Left Ideals of Finite Codimension}
\label{subsec:leftfin}

 By Corollary \ref{cor:ihardx}, 
  a two-sided ideal $I$ of $\chrisA$ has the form $\ihard{\pi}$
 for some  $\pi \in \Pi$ 
  (where $\ihard{\pi}:=\ihard{\{\pi\}}=\operatorname{ker} \pi$)
 if and only if $I$ is real and $\dim \chrisA/I < \infty$. 
 Lemma \ref{lem:ileftx} is the one-sided version of this fact.

\begin{lem}
\label{lem:ileftx}
 If $I$ is a  left ideal  of $\chrisA,$ then the  following are equivalent.
\begin{enumerate}[(1)] 
\item $I$ is real and $\dim \chrisA/I < \infty;$
\item There exist $\pi \in \Pi$ and $v \in V_{\pi}$ such that $I=\ileft{\pi,v}$ and $\pi(\chrisA)v=V_\pi$.
\end{enumerate}

 Moreover, if $(\pi,v)\in\cC_\lft$ is such that $\pi(\chrisA)v=V_\pi$ and if  
  $I\subset \icleft{\cC}{\pi,v}$ is  a two-sided ideal, then $I\subset \ichard{\cC}{\pi}.$
   In particular, if $\icleft{\cC}{\pi,v}$ is a two-sided ideal, then $\ichard{\cC}{\pi}=\icleft{\cC}{\pi,v}$.
\end{lem}

\begin{proof}
Clearly, (2) implies (1).  To prove the converse, 
 suppose  that $I$ is real and $\dim \chrisA/I < \infty$. % $I=\ileft{\pi,v}$
 Let $\pi$ be the left regular representation of $\chrisA$ on $\chrisA/I$ and $v=1+I$. 
 Let $V_\pi = \chrisA/I$ and note that  $\pi(\chrisA)v=V_\pi$.
 It remains to show that there exists an inner product on $\chrisA/I$ such that
 $\pi$ is a $\ast$-representation.

The kernel $Z(I)$ of the left regular representation $\pi$ is the
  largest two-sided ideal contained in $I$ by Lemma \ref{lem:zi0}. 
   Moreover, $\chrisA/Z(I)$ is isomorphic to a subspace
   of the linear maps on the finite dimensional space
   $\chrisA/I$.  Hence $\chrisA/Z(I)$ is finite dimensional. 
%Let $Z(I)$ be the maximal two-sided ideal contained in $I$. 
By Lemma \ref{lem:zi0} and  Lemma \ref{lem:zi}\eqref{it:zi1}, 
 $Z(I)$ is a real $\ast$-ideal.
%By the proof of Lemma \ref{hilba}, $Z(I)=\operatorname{Ker} \pi$,
%so that $\dim \chrisA/Z(I) < \infty$.
  By Lemma \ref{fr} and  Proposition \ref{structure},
$\chrisA/Z(I)$ is $\ast$-isomorphic to a finite direct sum $\oplus_i M_{n_i}(F_i)$
where $n_i \in \NN$, $F_i \in \{\RR,\CC,\mathbb{H}\}$ and the involution is conjugate transpose.

Let $\rho \colon \chrisA \to \oplus_i M_{n_i}(F_i)$ be the composition of the canonical mapping
$\chrisA \to \chrisA/Z(I)$ and the isomorphism $\chrisA/Z(I) \to \oplus_i M_{n_i}(F_i)$.
Since $\chrisA/I$ is isomorphic to $\rho(\chrisA)/\rho(I)$ as a vector space, it suffices to construct a positive 
linear functional $L$ on $\rho(\chrisA)$ such that $\rho(I)=\{b \in \rho(\chrisA) \colon L(b^*b)=0\}$.
It is well-known that every left ideal in a semisimple algebra is generated by an idempotent; see e.g. \cite[Corollary 2.1A]{faith}.
Therefore, $\rho(I)=\rho(\chrisA)e$ for some idempotent $e \in \rho(\chrisA)$ and we can take
\[
L(b):=\operatorname{trace}\big((1-e)^* b (1-e)\big)
\]
where $1$ is the identity matrix and $b$ runs through $\rho(\chrisA)=\oplus_i M_{n_i}(F_i)$.

 To prove the moreover statement,
  observe since  $\pi(\icleft{\cC}{\pi,v})v=0$ and
 also $I$ is a two-sided ideal,  that $\pi(I)\pi(\chrisA)v=0$. 
  Since $\pi(\chrisA)v= V_\pi$, it follows that $\pi(I)=0$.  Hence
  $I\subset \ichard{\cC}{\pi}.$  In the case that $\icleft{\cC}{\pi,v}$ is a
   two-sided ideal, $\icleft{\cC}{\pi,v}\subset \ichard{\cC}{\pi}.$ The reverse inclusion is evident
   and hence $\ichard{\cC}{\pi} =\icleft{\cC}{\pi,v}$.  
\end{proof}

\begin{rem}
For $\chrisA=\FA$, Lemma \ref{lem:ileftx} 
can also be deduced from \cite[Theorem 4.1]{chmn} and \cite[Lemma 3]{lewin}.
\qed \end{rem}

\subsection{Left Radical of a Two-Sided Ideal Is Two-Sided}
\label{subsec:leftishard}

 Note, if $I$ is a two-sided ideal in $\chrisA$, then 
\[ \rcleft{\cC}{I} \subseteq \rchard{\cC}{I}. \]
 We would like to know when the opposite inclusion holds.

We say that a representation class $\cC$ is \df{regular} if for every  $(\pi,v)
\in \cC_\lft$ there exists
$(\tilde{\pi},\tilde{v}) \in \cC_\lft$ such that
$\tilde{\pi}(\chrisA)\tilde{v}=V_{\tilde{\pi}}$ and 
$\Vert \pi(a)v \Vert= \Vert \tilde{\pi}(a) \tilde{v} \Vert$ for every $a \in
\chrisA$.   %Next, we need some natural examples of regular representation classes. 
Recall that $\Pi$ is the class of all finite-dimensional $\ast$-representations,
$\cB$ is the class of all bounded $\ast$-representations and $\cR$ is the class
of all $\ast$-representations.

\begin{prop}
\label{prop:regular}
The representation classes $\Pi$, $\cB$ and $\cR$ are regular.
\end{prop}

\begin{proof} Let $\cC$ be a representation class of $\chrisA$ and $(\pi,v) \in
\cC_\lft$.
Then $I:=\{a \in \chrisA \colon \pi(a)v=0\}$ is a left ideal of $\chrisA$. Let
$\tilde{\pi}$ be the left regular representation of $\chrisA$ on
$V_{\tilde{\pi}}:=\chrisA/I$.
We endow $\chrisA/I$ with the inner product $\langle a+I,b+I\rangle:=
\langle \pi(a)v,\pi(b)v \rangle_{V_\pi}$ so that $\tilde{\pi}$ becomes a
$\ast$-representation.
Let us define $\tilde{v}:=1+I$. Clearly, $\tilde{\pi}(\chrisA)=V_{\tilde{\pi}}$
and $\Vert \tilde{\pi}(a)\tilde{v} \Vert= \Vert a+I \Vert=\Vert \pi(a)v \Vert¡$
for every $a \in \chrisA$.

It remains to show that $\tilde{\pi} \in \cC$ if $\cC$ is one of the classes
$\Pi$, $\cB$.
If $\pi$ is finite-dimensional, then $\tilde{\pi}$ is also finite-dimensional
because 
$\dim \chrisA/\mathrm{Ker}\, \pi <\infty$ and $\mathrm{Ker}\, \pi \subseteq I$
implies $\dim \chrisA/I <\infty$.
If $\pi$ is bounded, then $\tilde{\pi}$ is also bounded because $\Vert
\tilde{\pi}(a)(b+I) \Vert=\Vert ab+I \Vert
= \Vert \pi(ab)v \Vert =\Vert \pi(a) \pi(b)v \Vert \le \vertiii{\pi(a)} \,\Vert
\pi(b)v \Vert
= \vertiii{\pi(a)} \, \Vert b+I \Vert$. 
\end{proof}

\begin{prop}
\label{prop:leftvstwo}
 If $\cC$ is a regular representation class of $\chrisA$ and 
 $I$ is a two-sided ideal in $\chrisA$,  then
\[ 
 \rchard{\cC}{I} \subseteq \rcleft{\cC}{I}.
\]
\end{prop}

\begin{proof}
Pick any $b \in \rchard{\cC}{I}$ and any $(\pi,v) \in \cC_\lft$ such that
$\pi(I)v=0$.
Let $(\tilde{\pi},\tilde{v}) \in \cC_\lft$ be such that
$\tilde{\pi}(\chrisA)\tilde{v}=V_{\tilde{\pi}}$ and 
$\Vert \pi(a)v \Vert= \Vert \tilde{\pi}(a) \tilde{v} \Vert$ for every $a \in
\chrisA$.
In particular, $\tilde{\pi}(I)\tilde{v}=0$. Since $I\subset \icleft{\cC}{\tilde{\pi},\tilde{v}}$ is a two-sided ideal, 
it follows from Lemma \ref{lem:ileftx} that $I\subset \ichard{\cC}{\tilde{\pi}}$; i.e.,
$\tilde{\pi}(I)=0.$   Now $b \in
\rchard{\cC}{I}$ implies that
$\tilde{\pi}(b)=0$. Hence, $\tilde{\pi}(b)\tilde{v}=0$ which implies that
$\pi(b)v=0$. This proves that $b \in \rcleft{\cC}{I}$.
\end{proof}

As an illustration of Propositions \ref{prop:leftvstwo} and \ref{prop:regular} 
we give an alternative proof of 
Proposition \ref{thm:mainfindim} \eqref{it:big}.
%Corollary \ref{cor:ihardx} for $\chrisA=\FA$. 

\begin{proof}
Let $I$ be a two-sided ideal in a $\chrisA=\FA$ such that $\dim \chrisA/I < \infty$. 
By \cite[Lemma 3]{lewin}, $I$ is finitely generated as a left ideal. 
Therefore, $\rr{I}=\rleft{I}=\rhard{I}$ where the first 
equality comes from \cite[Theorem 1.6]{chmn} and the second one 
from Propositions \ref{prop:leftvstwo} and \ref{prop:regular}.
By \cite[Theorem 4.1]{chmn} and the GNS construction, there exists $(\pi,v)\in \Pi_\lft$
such that $\rr{I}=\icleft{\cC}{\pi,v}$ and $\pi(\chrisA)v=V_\pi$. By the moreover
 portion of Lemma \ref{lem:ileftx},  $\icleft{\cC}{\pi,v}=\ichard{\cC}{\pi}$.
\end{proof}

\section{Soft Zeros}
\label{sec:softZeroes}

 A tuple $X$ in $\matof{}{\F}{g}$ is a \df{soft zero} 
 of a polynomial $p$ from $\F \axs$ if 
 $\det p(X)=0.$ 
 Replacing hard with soft zeros in the definitions 
 in Subsection \ref{sec:idealsinFA}
 produces the notions of the soft vanishing set, soft variety and soft radical. 
This also works for a general $\ast$-algebra $\chrisA$ and a general representation class $\cC$ of $\chrisA$.
We say that a ``real point'' $\pi \in \cC$ is a \df{soft zero} of a ``polynomial'' $a \in \chrisA$ 
if $\pi(a)$ is not invertible. Again, we can define the soft vanishing set, soft variety and soft radical.
We choose to work with general $\chrisA$ but only the with the simplest $\cC$, i.e. $\cC=\Pi$, the finite-dimensional $\ast$-representations.

Given a subset $T$ of $\Pi$, the \df{soft vanishing set} of $T$ is 
\[
  \isoft{T} =\{a \in \chrisA \colon \det \pi(a) =0, \mbox{ for all } \pi \in T\}.
\] \i{ $\isoft{T}$}
The set $\isoft{T}$ satisfies $\chrisA \, \isoft{T}  \chrisA \subseteq \isoft{T}$ but in general it is not closed under sums.
(Thus it is not an ideal.) Likewise, given a subset $\cS$ of $\chrisA$, the \df{soft variety} of $\cS$ is
\[
  \vsoft{\cS} = \{\pi \in \Pi \colon \det \pi(a) =0, \mbox{ for all } a \in \cS\}
\] \i{$ \vsoft{\cS}$}
and the \df{soft radical} of $\cS$ is
\[
  \rsoft{\cS} = \isoft{\vsoft{\cS}}.
\]
\i{$\rsoft{\cS}$}

For  $\pi \in \Pi$, it is convenient to abbreviate $\isoft{\{\pi\}}$ by $\isoft{\pi}$.
In subsection \ref{sec:structureSoft}, we describe the structure of $\isoft{\pi}$
and in subsection \ref{sec:softishard} we describe exactly when $\isoft{\pi}=\ihard{\pi}$.
This is used in subsection \ref{sec:softisideal} to characterize when 
$\isoft{\pi}$ is a two-sided ideal.

\subsection{The Structure of $\isoft{\pi}$}
\label{sec:structureSoft}

For a left ideal $I$ of $\chrisA,$ let
\[
\widehat{I} := \{p \in \chrisA \colon \text{ there exists } q \in \chrisA \setminus I \text{ such that } pq \in I\}.
\]

\begin{prop}
\label{prop:isoftx}
For a subset $\mathcal{S}$ of a $\ast$-algebra $\chrisA$ the following are equivalent:
\begin{enumerate}
\item $\mathcal{S}=\isoft{\pi}$ for some finite-dimensional $\ast$-representation $\pi$ of $\chrisA$.
\item $\mathcal{S}=\bigcup_{i=1}^k \widehat{I_i}$ for some $k \in \NN$ and some real left ideals $I_1,\ldots,I_k$ of $\chrisA$
with $\dim \chrisA/I_i < \infty$.
\end{enumerate}
\end{prop}

\begin{proof}
To prove that (1) implies (2), recall that every finite-dimensional $\ast$-representation is a finite direct sum of irreducible $\ast$-representations,
see e.g. \cite[Proposition 9.2.4]{palmer}. Furthermore, if $\pi=\oplus_i \pi$, then clearly $\isoft{\pi}=\bigcup_i \isoft{\pi_i}$.
Therefore, in view of Lemma \ref{lem:ileftx}, it suffices to show that for every irreducible $\ast$-representation $\pi$ of $\chrisA$
and every nonzero $w \in V_\pi$ we have that 
\[
\isoft{\pi}=\widehat{\ileft{\pi,w}}.
\]
Clearly, for each  $p \in \isoft{\pi}$ there exists a nonzero $v \in V_\pi$ such that $\pi(p)v=0$.
We claim that for each nonzero $w \in V_\pi$ there exists $q \in \chrisA$ such that $\pi(q)w=v$.
This claim implies that $\pi(p) \pi(q)w = 0$ and $\pi(q)w \ne 0$, so that $p \in \widehat{\ileft{\pi,w}}$.
We will prove the claim by contradiction. If $v \not\in \pi(\chrisA)w$, then $\pi(\chrisA)w$ is a proper 
nontrivial invariant subspace for $\pi(\chrisA)$. Now, \cite[Proposition 9.2.4]{palmer} implies that $\pi$ is reducible.

Suppose now that (2) is true. By Lemma \ref{lem:ileftx}, every $I_i$ is of the form $\ileft{\pi_i,v_i}$
for some finite-dimensional $\ast$-representation $\pi_i$ and some $v_i \in V_{\pi_i}$ such that $\pi_i(\chrisA)v_i=V_{\pi_i}$.
We claim that $\widehat{I_i}=\isoft{\pi_i}$. Namely, take any $p \in \chrisA$ and recall that $p \in \widehat{I_i}$ if and only if
$pq \in I_i$ for some $q \in \chrisA \setminus I_i$. The latter is true if and only if there exists $q \in \chrisA$ such that 
$\pi_i(q)v_i \ne 0$ and $\pi_i(p) \pi_i(q) v_i=0$ which is true if and only if there exists $w_i \in V_{\pi_i}$
such that $w_i \ne 0$ and $\pi_i(p)w_i=0$. The latter is equivalent to $p \in \isoft{\pi_i}$.
The claim implies (1) since 
$\cup_{i=1}^k \widehat{I_i}= \cup_{i=1}^k \isoft{\pi_i}=\isoft{\oplus_{i=1}^k \pi_i}$.
\end{proof}

\subsection{When $\isoft{\pi}$  Has the Form $\ihard{\psi}$ }
\label{sec:softishard}

\begin{prop}
 \label{lem:softhard}
 For a $\ast$-algebra $\chrisA$ and representation $\pi \in \Pi$ the following are equivalent:
\begin{enumerate}
\item \label{it:softhard1}
       $\isoft{\pi} \subseteq \ihard{\psi}$ for some $\psi \in \Pi$.
\item \label{it:softhard2}
      $\isoft{\pi} = \ihard{\pi}$.
\end{enumerate}
If (2) is true and $\F=\CC$, then $\pi(\chrisA)$ is $\ast$-isomorphic to $\CC$ endowed with the  standard involution.
If (2) is true and $\F=\RR$, then $\pi(\chrisA)$ is $\ast$-isomorphic 
to either $\RR$ or $\CC$ or $\mathbb{H}$ with standard involutions.
\end{prop}

\begin{proof}
First, $\ihard{\pi} \subset \isoft{\pi}$ by definition.
Next, suppose $a \in \isoft{\pi}$, which is equivalent to $\det \pi(a) = 0$.
Then $\det(\pi(a)^*\pi(a)) = 0$ as well, and since $\isoft{\pi} \subset
\ihard{\psi}$, we have $a^*a \in \ihard{\psi}$.  Further, $\pi(a)^*\pi(a)$
cannot have any nonzero eigenvalues $\lambda$ since, if it did, then $\lambda$
would be real,
$a^*a - \lambda \in \isoft{\pi} \subset \ihard{\psi}$, and so $a^*a -
(a^*a - \lambda) = \lambda \in \ihard{\pi}$.   Therefore $\pi(a)^*\pi(a) = 0$,
which implies that $\pi(a) = 0$.  Therefore $\isoft{\pi} \subset
\ihard{\pi}$, which implies that $\isoft{\pi} = \ihard{\pi}$.

Suppose that $\isoft{\pi} = \ihard{\pi}$ for some $\pi \in \Pi$.
Since $\pi(\chrisA)$ is contained in $M_n(\F)$ for some $n$, it is
a finite-dimensional formally real $\ast$-algebra. We claim
that $\pi(\chrisA)$ has no zero-divisors. Namely, if  
$\pi(a)\pi(b)=0$ for some $a,b \in \chrisA$, then either $\det \pi(a)=0$ or $\det \pi(b)=0$
which implies that either $\pi(a)=0$ or $\pi(b)=0$.  
The claims about the structure of $\pi(\chrisA)$ now follow from Proposition \ref{structure}.
\end{proof}

A representation $\pi$ is \df{unitarily equivalent} to representation $\psi$
if there exists a unitary operator $T \colon V_\pi \to V_\psi$ such that 
$\psi(a)=T \pi (a) T^\ast$ for every $a \in \chrisA$.

\begin{cor}
 \label{cor:softhard}
For $\ast$-algebra $\chrisA$ and every irreducible $\pi \in \Pi$ the following are equivalent.
\begin{enumerate}
\item $\isoft{\pi} = \ihard{\pi}$.
\item $\pi(\chrisA)$ has no zero divisors.
\item If $\F=\CC$, then $\pi$ is unitarily equivalent to some $\ast$-representation $\psi \colon \chrisA \to \CC$.
If $\F=\RR$, then $\pi$ is unitarily equivalent to some $\ast$-representation of one of the following types:
(i) $\psi \colon  \chrisA \to \RR$ where $\pi(\chrisA)=\RR,$ (ii) $\psi
\colon \chrisA \to M_2(\RR)$ where $\pi(\chrisA)=\CC$
or (iii) $\psi \colon \chrisA \to M_4(\RR)$ where $\pi(\chrisA)=\HH$.
\end{enumerate}
\end{cor}

\begin{proof}
In the proof of Proposition \ref{lem:softhard} we showed that (1) implies (2). To show that (3) implies (1)
note that every element $A \in \CC \subset M_2(\RR)$ satisfies $\det A=0$ if and only if
 $A=0$ and that every element
$B \in \HH \subset M_4(\RR)$ satifies $\det B=0$ iff $B=0$. Finally, (2) implies (3) by the 
Burnside's theorem for irreducible subalgebras of $M_n(\CC)$ and its real version \cite[Theorem 6]{lz}.
(Clearly, if $\pi$ is similar to $\psi$, then $\pi$ is also unitarily equivalent to $\psi$.)
\end{proof}

%\begin{rem}
% Proposition \ref{lem:softhard} implies the following:
%\begin{enumerate}[(i)]
%\item If $\isoft{\pi} = \ihard{\pi}$ and $\F=\CC$, then there exists a $\ast$-representation
%$\chi \colon \chrisA \to \CC$ such that $\isoft{\pi} = \ihard{\pi} = \ihard{\chi}$.
%\item If $\isoft{\pi} = \ihard{\pi}$ and $\F=\RR$, then there exists a $\ast$-representation
%$\chi \colon \chrisA \to M_4(\RR)$ such that $\chi(\chrisA)$ is contained in the span of matrices
%\begin{align}
%\notag
%%\label{eq:spanOfMats}
%\left(
%\begin{array}{cccc}
% 1&0&0&0\\
%0&1&0&0\\
%0&0&1&0\\
%0&0&0&1
%\end{array}
%\right),
%\left(
%\begin{array}{cccc}
% 0 & 1 & 0 & 0 \\
% -1 & 0 & 0 & 0 \\
% 0 & 0 & 0 & -1 \\
% 0 & 0 & 1 & 0 \\
%\end{array}
%\right),
%\left(
%\begin{array}{cccc}
% 0 & 0 & 1 & 0 \\
% 0 & 0 & 0 & 1 \\
% -1 & 0 & 0 & 0 \\
% 0 & -1 & 0 & 0 \\
%\end{array}
%\right),
%\left(
%\begin{array}{cccc}
% 0 & 0 & 0 & 1 \\
% 0 & 0 & -1 & 0 \\
% 0 & 1 & 0 & 0 \\
% -1 & 0 & 0 & 0 \\
%\end{array}
%\right)
%\end{align}
%and $\isoft{\pi} = \ihard{\pi} = \ihard{\chi}.$\qed 
%\end{enumerate}
%\end{rem}%
%
% ** END OLD **
%

\subsection{When $\rsoft{\{p\}}$ and $\rr{\chrisA p}$ Are Two-sided Ideals}
\label{sec:softisideal}

Given an  element $p$ of a $\ast$-algebra $\chrisA$ let  $J_p$ 
\i{$J_p$} and $I_p$
\i{$I_p$} denote the left and two-sided ideals generated by $p$ respectively. 
We will write $\irr$ for the set of all irreducible finite-dimensional $\ast$-representations of $\chrisA$. Recall that every $\pi \in \Pi$
can be decomposed as an orthogonal sum of finitely many elements from $\irr$. Recall also that $\pi(\chrisA)v=V_\pi$
for every $\pi \in \irr$ and every nonzero $v \in V_\pi.$

\begin{lem}
\label{lem:twojp}
For every element $p$ of a $\ast$-algebra $\chrisA$ we have that
$$\ihard{\vsoft{p} \cap \irr} \subseteq \rleft{J_p}.$$
If $\rleft{J_p}$ is a two-sided ideal, then  the opposite inclusion holds too.
\end{lem}

Recall that $\rleft{J_p}=\rr{J_p}$ if $\chrisA=\F \axs$ by the left Nullstellenatz \cite[Theorem 1.6]{chmn}.

\begin{proof}
Take any $q \in \ihard{\vsoft{p} \cap \irr}$ and any $(\pi,v) \in \Pi_\lft$ such that $\pi(p)v=0$.
Let us decompose $\pi=\pi_1 \oplus \ldots \oplus \pi_k$ where $\pi_i \in \irr$
and $v=v_1 \oplus \ldots \oplus v_k$ where $v_i \in V_{\pi_i}$. It follows that $\pi_i(p)v_i=0$
for every $i=1,\ldots,k$. Therefore, either $\det \pi_i(p)=0$ or $v_i=0$, which implies that
either $\pi_i(q)=0$ or $v_i=0$ for every $i=1,\ldots,k$. Consequently, $\pi_i(q)v_i=0$ for every
$i=1,\ldots,k$, and so, $\pi(q)v=0$. This proves the first part. 

To prove the second part, take any $r \in \rleft{J_p}$ any $\pi \in \vsoft{p} \cap \irr$.
Since $\det \pi(p)=0$, there exists a nonzero $v \in V_\pi$ such that $\pi(p)v=0$.
Since $\rleft{J_p}$ is a two-sided ideal, we have that $rs \in \rleft{J_p}$
for every $s \in \chrisA$. It follows that $\pi(rs)v=0$ for every $s \in \chrisA$.
Since $\pi(\chrisA)v=V_\pi$, it follows that $\pi(r)=0$. Therefore, $r \in \ihard{\vsoft{p} \cap \irr}$.
\end{proof}

\begin{prop}
For $p \in \F \axs$ the following are equivalent:
\begin{enumerate}
\item The ideal $\rr{J_p}$ is two-sided (i.e. $\rr{J_p}=\rr{I_p}$).
\item For every $\pi \in \irr$, $\det \pi(p)=0$ implies $\pi(p)=0$ (i.e. $\vsoft{p} \cap \irr \subseteq \vhard{p}$).
\end{enumerate}
\end{prop}

\begin{proof}
If $\vsoft{p} \cap \irr \subseteq \vhard{p}$, then clearly $\rr{I_p} \subseteq \ihard{\vhard{p}} \subseteq \ihard{\vsoft{p} \cap \irr}$.
By Lemma \ref{lem:twojp}, we have $\ihard{\vsoft{p} \cap \irr} \subseteq \rleft{J_p}=\rr{J_p}$.
It follows that $\rr{I_p} \subseteq \rr{J_p}$.

Conversely, if $\rr{I_p} = \rr{J_p}$, then $\rr{J_p}$ is a two-sided ideal by Lemma \ref{lem:zi}
and it is finitely generated as a left ideal by the Real Algorithm \cite[Theorem 3,1]{chmn}. Proposition \ref{thm:mainfindim} now implies that
$\rr{J_p}= \ihard{\vhard{p}}$. (Namely, by the first part of Proposition \ref{thm:mainfindim},
$\rr{J_p}$ has finite codimension. Therefore, the assumptions of the second part of Proposition \ref{thm:mainfindim}
are satisfied. ) On the other hand, we have that $\rr{J_p}=\ihard{\vsoft{p} \cap \irr}$ by Lemma \ref{lem:twojp}.
Finally, $$\vsoft{p} \cap \irr \subseteq \vhard{\ihard{\vsoft{p} \cap \irr}}=\vhard{\rr{J_p}}
=\vhard{\ihard{\vhard{p}}}=\vhard{p}$$ as claimed.
\end{proof}

\begin{lem}
\label{lem:twosoft}
For every element $p$ of a $\ast$-algebra $\chrisA$ we have that
$$\rleft{J_p} \subseteq \rsoft{\{p\}} = \isoft{\vsoft{p} \cap \irr}.$$
%If $\rsoft{\{p\}}$ is a two-sided ideal, then we also have the opposite inclusion. 
\end{lem}

\begin{proof}
Since $\vsoft{p} \cap \irr \subseteq \vsoft{p}$, we have that 
$\rsoft{\{p\}} =\isoft{\vsoft{p}} \subseteq \isoft{\vsoft{p} \cap \irr}$.
Conversely, take any $q \in \isoft{\vsoft{p} \cap \irr}$ and any $\pi \in \vsoft{p}$.
Let us decompose $\pi=\pi_1 \oplus \ldots \oplus \pi_k$ where $\pi_i \in \irr$ for all $i$.
Since $\det \pi(p)=0$, it follows that $\det \pi_i(p)=0$ for some $i$.
Therefore, $\det \pi_i(q)=0$, which implies that $\det \pi(q)=0$. 
This proves the equality.

To prove the inclusion 
%$\rleft{J_p} \subseteq \rsoft{\{p\}}$, 
take any $r \in \rleft{J_p}$ and any $\pi \in \vsoft{p}$.
Since $\det \pi(p)=0$, there exists a nonzero $v \in V_\pi$ such that $\pi(p)v=0$.
It follows that $\pi(r)v=0$ which implies that $\det \pi(r)=0$. Therefore, $r \in \rsoft{\{p\}}$.
\end{proof}

\begin{prop}
\label{prop:twosoft}
For every element $p$ of a $\ast$-algebra $\chrisA$, the following are equivalent.
\begin{enumerate}
\item The set $\rsoft{\{p\}}$ is a two-sided ideal.
\item For every $\pi \in \vsoft{p} \cap \irr$ we have that $\isoft{\pi} = \ihard{\pi}$.
(cf. Corollary \ref{cor:softhard}.)
\item $\rsoft{\{p\}}=\ihard{\vsoft{p} \cap \irr}.$
\item $\rsoft{\{p\}}=\rleft{J_p}$.
\end{enumerate}
\end{prop}

\begin{proof}
If (2) is false, then there exist $\pi \in \vsoft{p} \cap \irr$ and $q \in \chrisA$
such that $\det \pi(q) =0$ but $\pi(q) \ne 0$. Pick $w \in V_\pi$ such that $\pi(q)w \ne 0$.
For every $v$ in the unit sphere $S_\pi \subset V_\pi$, pick $r_v \in \chrisA$ such that $\pi(r_v)v=w$.
The sets $$U_v := \{u \in S_ \pi \colon \pi(q)\pi(r_v)u \ne 0\}$$ 
are clearly open and they cover $S_\pi$ because $v \in U_v$ for every $v \in S_\pi$. 
Since $V_\pi$ is finite-dimensional, $S_\pi$ is compact.
Pick $v_1,\ldots,v_k \in S_\pi$ such that $S_\pi=U_{v_1} \cup \ldots \cup U_{v_k}$
and consider the element $$r:= \sum_{i=1}^k r_{v_i}^\ast q^\ast q r_{v_i}.$$
By construction, $\langle \pi(r)v,v \rangle =\sum_{i=1}^k \Vert \pi(q) \pi(r_{v_k}) v \Vert^2 >0$
for every $v \in S_\pi$. Therefore $\det \pi(r) \ne 0$, and so $r \not\in \isoft{\vsoft{p} \cap \irr}$.
Since $q \in \rsoft{\{p\}}$ and $r \not\in \rsoft{\{p\}}$, it follows that (1) is false.

If (2) is true, then (3) follows from
$$\rsoft{\{p\}}=\isoft{\vsoft{p} \cap \irr}=
\!\!\bigcap_{\pi \in \vsoft{p} \cap \irr} \isoft{\pi}=
\!\!\bigcap_{\pi \in \vsoft{p} \cap \irr} \ihard{\pi}=\ihard{\vsoft{p} \cap \irr}.$$
If (3) is true, then (4) follows from Lemmas \ref{lem:twojp} and \ref{lem:twosoft}. Namely,
$$\rsoft{\{p\}} =\ihard{\vsoft{p} \cap \irr} \subseteq \rleft{J_p} \subseteq \rsoft{\{p\}}$$
shows that $\rleft{J_p}=\rsoft{\{p\}}$. 

Clearly, $\rsoft{\{p\}}$ is a two-sided semigroup ideal 
w.r.t. multiplication  in $\chrisA$. %$(\chrisA,\cdot)$. 
If (4) is true, then it is also a subgroup 
w.r.t. addition  in $\chrisA$.
%in $(\chrisA,+)$. 
Hence, (1) is true.
\end{proof}

\begin{rem}
If $\chrisA=\F \axs$, then $\rsoft{\{p\}}$ is a two-sided ideal
if and only if $\rr{J_p}=\rsoft{\{p\}}$ (since $\rleft{J_p}=\rr{J_p}$ by the left Nullstellenatz \cite[Theorem 1.6]{chmn})
if and only if $\rsoft{\{p\}}=\ihard{\psi}$ for some $\psi \in \Pi$ (by the Real Algorithm \cite[Theorem 3,1]{chmn} and Proposition \ref{thm:mainfindim}).
However, $\ihard{\psi}$ may be different from $\isoft{\psi}$ in this case because $\psi$ may not be irreducible.
\end{rem}

\begin{example}
 If $\chrisA = \CC \axs$,
then $\ihard{X} = \isoft{X}$ if and only if $\ihard{X}
= \ihard{(\lambda_1, \ldots, \lambda_g)}$ for some 
$\lambda_1, \ldots, \lambda_g \in \CC$. The polynomial $p$ defined by
\[
 p = \sum_{i=1}^g (x_i - \lambda_i)^*(x_i - \lambda_i) + \sum_{i=1}^g (x_i -
\lambda_i)(x_i - \lambda_i)^*
\]
satisfies $\rr{J_p} = \ihard{(\lambda_1, \ldots,\lambda_g)} = \isoft{(\lambda_1, \ldots,
\lambda_g)}$ and so $\rsoft{\{p\}}$ is a two-sided ideal.
\end{example}

\printindex

\end{document}